\documentclass[12pt]{amsart}
\usepackage[utf8x]{inputenc}
\usepackage{amssymb}
\usepackage{amsmath}
\usepackage{amsthm}
\usepackage[mathscr]{euscript}
\usepackage[normalem]{ulem}
\usepackage{hyperref}
\usepackage{appendix}
\usepackage{tikz}
\usepackage{tikz-cd}
\usepackage{verbatim}
\usepackage{color}
\usepackage{mathrsfs}

\usepackage[margin=3 cm,heightrounded=true,centering]{geometry}
\usepackage{graphicx}
\usepackage{wrapfig}
\usepackage{cancel}

\numberwithin{equation}{section}

\newtheorem{theorem}{Theorem}[section]
\newtheorem{proposition}[theorem]{Proposition}
\newtheorem{lemma}[theorem]{Lemma}
\newtheorem{remark}[theorem]{Remark}
\newtheorem{corollary}[theorem]{Corollary}
\newtheorem{definition}[theorem]{Definition}

\newtheorem{question}[theorem]{Question}

\DeclareMathOperator{\ric}{Ric}

\DeclareMathOperator{\id}{id}

\title[Conformal boundary rigidity from null geodesic travel times]{Conformal boundary rigidity from null geodesic travel times}

\author{Gabriel P. Paternain}
\address[Gabriel P. Paternain]{Department of Mathematics, University of Washington, Seattle, WA 98195-4350, USA}
\email{gpp24@uw.edu}

\author{Eric Woolgar}
\address[Eric Woolgar]{Department of Mathematical and Statistical Sciences and Theoretical Physics Institute, University of Alberta, Edmonton, Alberta, Canada T6G 2N8}
\email{ewoolgar@ualberta.ca}

\begin{document}

\date{\today}

\begin{abstract}
The gravitational field of a distant, isolated system is manifested by the conformally invariant Weyl tensor. Thus the conformal structure far from the system encodes the system's gravitational mass. It also encodes the causal structure, thereby linking it to the mass. For asymptotically anti-de Sitter (AdS) spacetimes, this link led to a novel positive mass theorem of Page, Surya, and the second author \cite{PSW} which did not rely on any traditional energy condition. Here we ask whether that theorem has a rigidity case. Specifically, we consider all null geodesics in an asymptotically AdS spacetime that depart from the Penrose conformal infinity, travel through spacetime, and return to conformal infinity. If all such geodesics from a given point refocus at an antipodal point at infinity, is the spacetime conformal to anti-de Sitter space? It is easy to answer the question if the asymptotically AdS spacetime either (i) obeys the null energy condition in 3 or 4 spacetime dimensions, or (ii) is static (in any dimension), and we give simple proofs in those cases. We then answer the question in the case of globally stationary, asymptotically AdS spacetimes, by applying the theory of magnetic geodesics on the Riemannian manifold-with-boundary obtained by quotienting by the stationary Killing vector field. The question has an analogue for asymptotically flat spacetimes, which we also discuss.
\end{abstract}

\maketitle

\section{Introduction}\label{section1}
\setcounter{equation}{0}

\noindent It is well-known that light travel times are affected by the presence of a massive object. This is the \emph{time delay of light} effect. The time delay for light rays passing near the Sun has been measured and used as observational verification of the general theory of relativity. A natural question is to ask in which spacetimes do light rays experience no time delay. We pose (and in some special cases answer) versions of this question below.

Asymptotically anti-de Sitter spacetimes are $(n+1)$-dimensional spacetimes, $n\ge 2$, whose curvatures asymptote to that of a vacuum constant curvature spacetime with negative cosmological constant. Sectional curvatures approach $-1$ at infinity. Such spacetimes admit a Penrose conformal compactification such that conformal infinity has the topology ${\mathbb R}\times {\mathbb S}^{n-1}$ and has an induced conformal class of metrics containing a representative $-dt^2\oplus g({\mathbb S}^{n-1},{\rm can})$ where $g({\mathbb S}^{n-1},{\rm can})$ means the round or canonical metric on the $(n-1)$-sphere. Thus the conformal boundary is a timelike ``cylinder'' in the conformally extended spacetime. The construction is standard; see for example \cite[Section 11.1]{Wald} (which, however, emphasizes primarily the asymptotically flat version wherein the conformal boundary is a double-napped cone).

The conformal boundary of an asymptotically anti-de Sitter spacetime is therefore a conformal class of $n$-dimensional spacetimes in its own right. In particular, it has a causal structure. We will denote it by
\begin{equation}
\label{eq1.1}
{\mathscr{I}}:= \left ( {\mathbb R}\times {\mathbb S}^{n-1}, \left [-dt^2\oplus g({\mathbb S}^{n-1},{\rm can}) \right ] \right ),
\end{equation}
where the square brackets denote the conformal class. Considered as a spacetime in its own right, ${\mathscr{I}}$ has no time delay effect. More precisely, identify ${\mathscr{I}}$ with ${\mathbb R}\times{\mathbb S}^{n-1}\ni (t,\theta_1,\dots,\theta_{n-1})$ and let $p\in {\mathscr{I}}$. For the moment, choose the metric to be $-dt^2\oplus g({\mathbb S}^{n-1},{\rm can})$. Via isometries, we may place $p$ at one pole of the ${\mathbb S}^{n-1}$ factor at some time $t=0$ so $p=(0,0,\dots,0)$. Then every null geodesic of ${\mathscr{I}}$ from $p$ reconverges with every other such null geodesic at a point $q$ which lies at the opposite pole to $p$ in ${\mathbb S}^{n-1}$ and has $t$-coordinate $t=\pi$ (recalling that the canonical metric in any dimension is a round metric with sectional curvatures $1$ and, therefore, diameter $\pi$). Letting $\theta_1$ be the polar angle, we have $q=(\pi,\pi,0,\dots,0)$. Since conformal transformations preserve causal structure, the pairing of $p$ with $q$ does not depend on our choice of representative metric $-dt^2\oplus g({\mathbb S}^{n-1})$ within its conformal class. We will say that any such pair $p$ and $q$ are \emph{spacetime antipodes on ${\mathscr{I}}$}. Note that spacetime antipodes on ${\mathscr{I}}$ are defined using \emph{only the causal structure on the conformal infinity ${\mathscr{I}}$} of an asymptotically anti-de Sitter spacetime.

\begin{definition}[Timelike boundary Penrose property and holographic causality]\label{definition1.1}
Let $M$ be an asymptotically anti-de Sitter spacetime with conformal infinity ${\mathscr{I}}$, and let ${\bar M}=M\cup{\mathscr{I}}$. Then $M$ obeys the \emph{timelike boundary Penrose property} \cite{Cameron} if each pair of spacetime antipodes on ${\mathscr{I}}$ can be joined by a timelike curve. If no pair of spacetime antipodes on ${\mathscr{I}}$ can be joined by a timelike curve, we say that the spacetime has the \emph{holographic causality} property \cite{PSW}.
\end{definition}

It is clear that no timelike curve on ${\mathscr{I}}$ joins any antipodal pair of points. It is also clear that if there is a timelike curve from $p$ to $q$, then there will be a timelike curve that contains no points of ${\mathscr{I}}$ other than its endpoints at $p$ and $q$, so on any open interval in the domain of its parameter such a curve is a timelike curve in $M$. Such curves are sometimes said to be curves in spacetime $M$ with \emph{idealized endpoints} $p\in{\mathscr{I}}$ and $q\in{\mathscr{I}}$.

Positive mass asymptotically AdS spacetimes do not obey the timelike boundary Penrose property \cite{Woolgar}. Indeed, to prove a positive mass theorem in this setting, one can replace the usual assumption of a (conformally non-invariant) pointwise energy condition by the (conformally invariant) assumption that holographic causality holds \cite{PSW}. It is well-known that positive mass theorems usually imply that zero-mass manifolds are \emph{rigid}, meaning that they are single points in the space of metrics (modulo diffeomorphisms that preserve the asymptotic conditions listed in Definition \ref{definition2.1}). In the case of the theorem of \cite{PSW}, one might ask whether the conformal class of the anti-de Sitter metric arises as a type of rigidity (of conformal classes rather than of metrics). Specifically, we pose the following question.

\begin{question}\label{question1.2}
Let $M$ be an 
asymptotically AdS spacetime with conformal boundary ${\mathscr{I}}$. Say that for every $p\in {\mathscr{I}}$, every null geodesic in $M$ with past (idealized) endpoint at $p$ has future (idealized) endpoint at $q\in {\mathscr{I}}$, which is therefore the spacetime antipode for $p$. Is the spacetime conformally isometric to anti-de Sitter space?
\end{question}

The question echoes two well-known problems in Riemannian geometry. One is Blaschke's \emph{wiedersehenmannigfaltigkeit} or \emph{returning manifold} problem (see \cite[p 493]{Berger} for a discussion), which asks if the sphere is the only manifold such that every geodesic leaving any given initial point returns to that point's antipode. The other is the \emph{boundary rigidity problem}, a famous inverse problem which asks whether an unknown Riemannian metric on the interior of a manifold-with-boundary can be discerned from information about the lengths of geodesics that begin and end on the boundary. In the case of asymptotically hyperbolic manifolds with a conformal boundary-at-infinity, the analogous rigidity problem has recently been addressed in \cite{GGSU} by considering the renormalized lengths of complete (infinitely long) geodesics. A consequence of \cite[Theorem 2]{GGSU} is that in an asymptotically hyperbolic manifold, if every complete geodesic has the same renormalized length then there is a diffeomorphism between the manifold and standard hyperbolic space that is an isometry up to corrections of infinite order in the defining function for the conformal boundary-at-infinity.

Viewed as an inverse problem, Question \ref{question1.2} asks only about a single case with very special boundary data. It may yield to geometric considerations without resort to the full analytical machinery of inverse problems. Indeed, in special cases we will use geometric arguments to answer this question in the affirmative. One case will be when the conformal class of a 4-dimensional spacetime metric admits an asymptotically anti-de Sitter representative in which the null energy condition holds. Then the spacetime with this representative metric will be isometric to anti-de Sitter space. The other case will be when the spacetime (of 3 or more dimensions) is asymptotically anti-de Sitter and static (so that it admits a timelike, hypersurface-orthogonal Killing vector field). We will then generalize this result to the stationary case (where the timelike Killing vector field need not be hypersurface-orthogonal). The main result of this work is a partial answer to our question given as Theorem \ref{theorem4.4}, which we summarize here as follows.
\begin{theorem}
\noindent \emph{Question \ref{question1.2} is answered in the affirmative when the spacetime is globally stationary.}
\end{theorem}
But the question is intended to be general. Can it be answered without invoking any assumption of time symmetry, any energy condition, or any other restrictive assumption?

We only expect to determine the metric up to conformal isometries that preserve the asymptotically AdS property because we have formulated the question using only the causal structure of spacetime. To motivate this, recall that all asymptotically AdS metrics that are conformal to anti-de Sitter spacetime have zero Ashtekar--Magnon mass \cite{AM} since this mass vanishes whenever the Weyl tensor falls off quickly enough (which it trivially does for AdS since it vanishes identically in that case). The vanishing of Ashtekar--Magnon mass is a conformally invariant property. From this perspective, the question posed above is a question concerning rigidity for the positive mass theorem of \cite{PSW}. That theorem does not assume a pointwise energy condition. In place of an energy condition, the theorem of \cite{PSW} assumes holographic causality. Then we may ask what happens when an asymptotically AdS spacetime lies at the border between holographic causality and the Penrose property, in the sense that every null geodesic from each $p\in {\mathscr{I}}$ arrives at the spacetime antipode of $p$, giving rise to Question \ref{question1.2}.

In Subsection \ref{subsection2.1} we recall the definition of asymptotically anti-de Sitter spacetimes and prove a simple lemma. Subsection \ref{subsection2.2} reviews the definitions of stationary and static spacetimes. In Section \ref{section3} we give the two proofs of the affirmative answer to our question under the further assumptions either that (i) the metric is asymptotically simple, (3- or) 4-dimensional, and conformally related to a metric for which the null energy condition holds, or (ii) that spacetime is static. These arguments are mostly elementary.

In Section \ref{section4}, we generalize from static to stationary asymptotically anti-de Sitter spacetimes, answering the question in the affirmative in this more general case. Our method is to map null geodesics in spacetime to so-called magnetic geodesics on the quotient manifold obtained by equating spacetime points that lie on the same orbit of the stationary Killing vector field. Travel times for the null geodesics map to values of a line integral, the so-called Ma\~n\'e action, along magnetic geodesics. We then study the first variation formula for the Ma\~n\'e action, or in fact the related energy integral, of magnetic geodesics that join boundary points of this quotient (optical) manifold. The result, which is of independent interest, is that if every inextendible magnetic geodesic on a compact Riemannian manifold with totally geodesic round sphere boundary has the same Ma\~n\'e action (defined in equation \eqref{eq3.7}), then the optical manifold is a standard round hemisphere. This is Theorem \ref{theorem4.1} and Corollary \ref{corollary4.3}. Interestingly, while Question \ref{question1.2} for stationary spacetimes would be fully answered by assuming that the magnetic geodesics leaving a given boundary point in the quotient (Riemannian) manifold return to the boundary at the initial point's antipode, we do not need to assume this to prove Theorem \ref{theorem4.1} and Corollary \ref{corollary4.3}; that the returning magnetic geodesic reaches the antipode follows as a consequence of the constancy of the Ma\~n\'e action. The main result of the paper, Theorem \ref{theorem4.4}, then follows easily. These results are found in Subsection \ref{section4.2}, while the proof of Theorem \ref{theorem4.1} is postponed to Subsection \ref{section4.6}. Other subsections contain background material and preliminary results.

Section \ref{section5} poses the analogous question to Question \ref{question1.2} for asymptotically flat spacetimes and answers it affirmatively when the metric is conformally related to a metric for which the null energy condition holds.
For the reader's convenience and to make the article more self-contained, an appendix gives a derivation of the fact that null geodesics in stationary spacetimes map to magnetic geodesics in the quotient space (and to actual geodesics when the spacetime is static). This result is not easy to find in the recent mathematical relativity literature, but appears elsewhere in the literature in \cite{Germinario}, \cite[Section 4.2]{Stefanov}, \cite{OPS}; see also \cite{GHWW09, CaponioJavaloyesMasiello2011, CaponioJavaloyesSanchez2011, CaponioGerminarioSanchez2016,
HerreraJavaloyes2021} where the study of null geodesic motion in stationary spacetimes is recast in the language of Finsler geometry. A second appendix contains the proofs of two simple lemmata used in Section \ref{subsection3.1}.

\subsection*{Acknowledgments} GPP was supported by National Science Foundation (US) grant DMS--2347868. The research of EW is supported by Natural Sciences and Engineering Research Council of Canada Discovery Grant RGPIN--2022--03440. EW is grateful to Spyros Alexakis for a discussion of the static case, to Greg Galloway for a comment concerning the null energy condition proof, and to Peter Cameron and Giulio Sanzeni for discussions. We are also grateful to Erasmo Caponio for alerting us of a mistake in an earlier draft.

\section{Background}\label{section2}
\setcounter{equation}{0}

\subsection{Asymptotic conditions}\label{subsection2.1}

\begin{definition}[Asymptotically AdS spacetime]\label{definition2.1}
We will say that a spacetime $(M,g)$ is \emph{asymptotically anti-de Sitter} if there is a spacetime manifold $({\tilde M}, {\tilde g})$ with closure ${\bar M}={\tilde M}\cup {\mathscr{I}}$, ${\mathscr{I}}=\partial {\bar M}$, a diffeomorphism $\psi:{\tilde M}\to M$, and a function $u \ge 0$ on ${\bar M}$ such that
\begin{itemize}
\item [i)] $\ric_g=-ng+{\mathcal O}(u^2)$,
\item [ii)] ${\tilde g}=u^2 \psi^*g$,
\item [iii)] $u(p)=0$ iff $p\in \partial {\bar M}$,
\item [iv)] $du\neq 0$ on $\partial {\bar M}$ (for $d$ the exterior derivative on ${\bar M}$), and
\item [v)] ${\tilde g}$ extends twice differentiably to ${\mathscr{I}}$ and induces on ${\mathscr{I}}$ a conformal equivalence class of Lorentzian metrics $\left [-dt^2+g({\mathbb S}^{n-1},{\rm can})\right ]$.
\end{itemize}
Furthermore, an asymptotically AdS (or asymptotically flat; see Section \ref{section5}) spacetime is \emph{asymptotically simple}
\footnote{Asymptotically simple manifolds \cite[p 222]{HE} are not to be confused with \emph{simple (Riemannian) manifolds}, a completely distinct notion.}
: every null geodesic in $(M,g)$ extends to have two endpoints (one past and one future) on ${\mathscr{I}}$.
\end{definition}
It follows from point (i) and the differentiability assumption in point (v) of the definition that $|du|_{\tilde g}=1$, and that $u$ can be chosen such that ${\mathscr{I}}$ is totally geodesic in the conformal metric ${\tilde g}$. Also, every null geodesic in $(M,g)$ is, up to reparametrization, an incomplete null geodesic in $({\tilde M},{\tilde g})$. When we say that a null geodesic in $(M,g)$ extends to acquire endpoints on ${\mathscr{I}}$, we mean that it extends when considered as an incomplete geodesic in $({\tilde M},{\tilde g})$. Then we have the following simple lemma.

\begin{lemma}[Interiors of null geodesics do not meet ${\mathscr{I}}$]\label{lemma2.2}
Say that ${\tilde \gamma}$ is a null geodesic in the conformal metric ${\tilde g}=u^2 g$ that appears in the definition of an asymptotically anti-de Sitter spacetime. Further assume that every null geodesic from $p\in{\mathscr{I}}$ passes through the spacetime antipode $q\in{\mathscr{I}}$ of $p$. Then either ${\tilde \gamma}$ lies entirely on ${\mathscr{I}}$ or it does not meet ${\mathscr{I}}$ at any point between $p$ and $q$.
\end{lemma}

\begin{proof}
First, ${\tilde \gamma}$ cannot meet ${\mathscr{I}}$ tangentially unless it lies entirely on ${\mathscr{I}}$ because ${\mathscr{I}}$ is totally geodesic with respect to ${\tilde g}$. Next, it cannot meet ${\mathscr{I}}$ transversely at any point $r\in {\mathscr{I}}$ between $p$ and $q$ since, if it did, it would still have to be part of a null geodesic that reaches $q$. But any curve meeting $r$ transversely and continuing on to $q$ would necessarily have discontinuous tangent, and any causal curve with discontinuous tangent is not a geodesic (and could be deformed to a smooth timelike curve from $p$ to $q$, placing $q$ in the timelike future of $p$, contrary to assumption).
\end{proof}

This lemma shows that although spacetime antipodes on ${\mathscr{I}}$ are defined using the conformal metric on ${\mathscr{I}}$, they can also be defined using only null geodesics in spacetime, since every inextendible null geodesic in spacetime with past idealized endpoint $p$ will have future idealized endpoint $q$. This will be useful in Section \ref{section5}.

\subsection{Stationary and static spacetimes}\label{subsection2.2}
A spacetime is \emph{stationary} if there is a nowhere vanishing timelike Killing vector field (a vector field whose orbits are timelike curves along which translation is an isometry) and is \emph{static} if this vector field is hypersurface-orthogonal. More precisely, we make the following definition.

\begin{definition}[Asymptotically AdS stationary spacetime]\label{definition2.3}
An asymptotically anti-de Sitter spacetime is \emph{globally stationary} if it satisfies the conditions of Definition \ref{definition2.1} and possesses a globally defined timelike Killing vector field $\frac{\partial}{\partial t}$ with norm $N>0$, $\frac{\partial N}{\partial t}=0$.
\end{definition}

\begin{definition}[Fermat (optical) manifold]\label{definition2.4}
The \emph{quotient manifold} $\Sigma$ obtained by quotienting a stationary spacetime by orbits of the timelike Killing vector field $\frac{\partial}{\partial t}$ is also called the \emph{Fermat manifold} or the \emph{optical manifold}.
\end{definition}

Let $\varphi:M\to \Sigma$ be the quotienting map, so it is constant on integral curves of $\frac{\partial}{\partial t}$. We can give $\Sigma$ a Riemannian metric $h$ defined so that the pullback obeys $\left (\varphi^* h\right )\left ( \frac{\partial}{\partial t}, \cdot \right )=0$. In a convenient abuse of notation, we will drop the pullback and treat $h$ as both a Riemannian metric on $\Sigma$ and a symmetric $(0,2)$-tensor of rank $n$ on $M$. Similarly, there is a 1-form $\omega$ on $\Sigma$ defined so that its pullback obeys
\begin{equation}
\label{eq2.1}
-N^2\left ( dt+\varphi^*\omega\right )=\left ( \frac{\partial}{\partial t}\right )^{\flat}=g\left ( \cdot, \frac{\partial}{\partial t}\right ) .
\end{equation}
That is, $-N^2\left ( dt+\varphi^*\omega\right )$ is the 1-form metric-dual to $\frac{\partial}{\partial t}$. Once again, we will ignore the pullback and treat $\omega$ as both a 1-form on $\Sigma$ and on $M$. Note that on $M$, then $\pounds_{\frac{\partial}{\partial t}}\omega=0$ and $\pounds_{\frac{\partial}{\partial t}}h=0$. The spacetime metric takes the form
\begin{equation}
\label{eq2.2}
\begin{split}
ds^2 =&\, -N^2 \left ( dt+\omega\right )^2 +h\\
=&\,-N^2dt^2-N^2\left ( \omega\otimes dt +dt\otimes\omega\right )+\left ( h-N^2\omega\otimes\omega\right ) .
\end{split}
\end{equation}

\begin{definition}[Fermat (optical) metric]\label{definition2.5}
On the optical manifold the Riemannian metric ${\tilde h}:=h/N^2$ is called the \emph{Fermat metric} or the \emph{optical metric}.
\end{definition}

\begin{remark}\label{remark2.6}
Any such spacetime has the following properties.
\begin{itemize}
\item [i)] Since $N=\left \| \frac{\partial}{\partial t}\right \|$ does not vanish, stationary spacetimes with ergoregions are excluded from consideration.
\item [ii)] The metric $d{\tilde s}^2=ds^2/N^2$ extends to conformal infinity and induces the first fundamental form $-dt^2\oplus {\hat h}$ there, where ${\hat h}$ is the restriction of $h/N^2$ to $T{\partial M}$).
\item [iii)] From Definition \ref{definition2.1}{\rm{.(v)}}, the second fundamental form of ${\mathscr{I}}$, as defined by $d{\tilde s}^2$, vanishes so ${\mathscr{I}}$ is totally geodesic with respect to the conformally rescaled spacetime metric.
\end{itemize}
\end{remark}

\begin{remark}\label{remark2.7}
Since ${\mathscr{I}}$ is totally geodesic with respect to $d{\tilde s}^2$, then $\partial \Sigma$ is totally geodesic with respect to the optical metric ${\tilde h}$.
\end{remark}

Finally, we make the following definition.
\begin{definition}[Asymptotically AdS static spacetime]\label{definition2.8}
An asymptotically anti-de Sitter spacetime is \emph{globally static} if it is globally stationary and the timelike Killing vector field $\frac{\partial}{\partial t}$ in Definition \ref{definition2.3} is hypersurface-orthogonal.
\end{definition}

The optical manifold has boundary homeomorphic to ${\mathbb S}^{n-1}$, so a collar neighbourhood of the boundary is homotopy-equivalent to ${\mathbb S}^{n-1}$ and is hence simply connected for $n\ge 3$. So is the domain in spacetime which maps to it under the quotient. The 1-form metric-dual to $\frac{\partial}{\partial t}$ is $\left ( \frac{\partial}{\partial t}\right )^{\flat} =-N^2\left ( dt+\omega\right )=:X$. Hypersurface-orthogonality implies that $0=X\wedge dX = N^4 (dt+\omega)\wedge d\omega$ so $d\omega=0$. On a simply connected domain, then $\omega$ is exact, so $dt+\omega=d(t+f)$ for some function $f$ (which, by stasis, may be taken to be $t$-independent). We can then redefine the time coordinate to absorb $f$; equivalently, we can set $\omega=0$ on a simply connected collar neighbourhood of ${\mathscr{I}}$ in spacetime. On this neighbourhood, the static metric becomes
\begin{equation}
\label{eq2.3}
ds^2 =-N^2 dt^2 \oplus h .
\end{equation}
In the absence of event horizons, simple connectivity of the entire spacetime follows from the null energy condition and topological censorship \cite{Galloway, GSWW1, GSWW2}, but we do not wish to assume any energy condition when we treat the static case in Subsection \ref{subsection3.2} below.

\section{Two answers to Question \ref{question1.2}}\label{section3}

\subsection{The answer assuming the null energy condition}\label{subsection3.1}
We begin by assuming that there is a 4-dimensional asymptotically anti-de Sitter spacetime which obeys the null energy condition and which belongs to the conformal class of the spacetime metric. We will prove that the spacetime metric $g$ is the AdS metric. Then it is obvious that the spacetime will be conformally isometric to anti-de Sitter space.

\begin{proposition}[Null energy condition rigidity]\label{proposition3.1}
Let $(M,g)$ be a 4-dimensional, asymptotically simple, asymptotically anti-de Sitter spacetime with conformal infinity ${\mathscr{I}}$ as in Definition \ref{definition2.1}. Assume that the null energy condition
\begin{equation}
\label{eq3.1}\ric(\ell,\ell)\ge 0
\end{equation}
holds for all null vectors $\ell\in T_rM$ and all $r\in M$. Assume also that holographic causality holds and, for each $p\in {\mathscr{I}}$, every null geodesic from $p$ passes through its spacetime antipode $q\in {\mathscr{I}}$. Then $(M,g)$ is anti-de Sitter spacetime.
\end{proposition}

\begin{proof}
Along a null geodesic, we can define its expansion scalar $\theta$, its shear $\sigma$, and its rotation $\omega$. We will assume the reader is familiar with the details, which can be found in \cite[Chapter 4]{HE} or \cite[Chapter 12]{BEE}.

A geodesic congruence that issues from a single point cannot have nonzero rotation parameter. The congruence of null geodesics issuing from $p\in {\mathscr{I}}$ has zero rotation parameter as computed using ${\tilde g}$. This property is conformally invariant, so the congruence has zero rotation parameter in $(M,g)$.

The equations governing this null geodesic congruence as it traverses through $(M,g)$ are therefore
\begin{equation}
\label{eq3.2}
\begin{split}
\frac{d\theta}{ds}=&\, -\ric(\ell,\ell)-2|\sigma|^2-\frac{1}{(n-1)}\theta^2,\\
\frac{d\sigma}{ds}=&\, -C(\cdot,\ell,\cdot,\ell) -\theta\sigma-\sigma\cdot\sigma+|\sigma|^2\id,
\end{split}
\end{equation}
where $C$ is the $(0,4)$-Weyl tensor, $s$ is an affine parameter along each geodesic (so $\ell=\frac{d}{ds}$) and $\id$ is the identity on the reduced tangent space spanned by Jacobi fields.

The reader will recall that spacetime antipodes on ${\mathscr{I}}$ are defined using only the causal structure of ${\mathscr{I}}$. But if every null geodesic \emph{in spacetime} from $p\in {\mathscr{I}}$ returns to ${\mathscr{I}}$ at the spacetime antipode $q\in {\mathscr{I}}$ and if holographic causality holds, then no causal curve in spacetime is ``faster'' (by definition of spacetime antipode on ${\mathscr{I}}$, no causal curve in ${\bar M}$ is faster either). Then $q\in \partial I^+(p)$ (in ${\bar M}$). Now no null geodesic lying along $ \partial I^+(p)$ can have conjugate points, and by assumption we have a congruence of null geodesics in $M$ from $p$ to $q$, so these curves cannot have conjugate points. But if $\ric(\ell,\ell)\ge 0$, a standard argument based on the first equation in \eqref{eq3.2} shows that we must have $\ric(\ell,\ell)=0$ and $\sigma=0$ all along every such null geodesic. Furthermore, from the second equation in \eqref{eq3.2}, this implies that $-C(\cdot,\ell,\cdot,\ell)=0$ all along each such null geodesic as well.

Now choose an arbitrary point $r\in M$ and fix a future-null direction $\ell_r$ at $r$. Exponentiating, we obtain a null geodesic through $r$ and by assumption this geodesic can be extended to meet ${\mathscr{I}}$ at two endpoints, say $p$ to the past and $q$ to the future, and also by assumption every null geodesic from $p$ ends at $q$ and $q$ is the spacetime antipode on ${\mathscr{I}}$ of $p$. Then by the argument above, $\ric(\ell_r,\ell_r)=0$ and $C(v_r,\ell_r,w_r,\ell_r) =0$ (for arbitrary vectors $v_r, w_r\in T_rM^{\perp}$ in the orthogonal complement of $\ell$). But since $\ell_r$ was an arbitrary null direction, Lemma \ref{lemma3.2}
below shows that $\mathrm{Ric}$ is proportional to $g$ at $r$, and
Lemma \ref{lemma3.3} below shows that $C=0$ at $r$.
Since $r$ was arbitrary, $g$ must be Einstein and locally conformally flat. Now any metric that is locally conformally flat and Einstein has constant curvature, and since the conformal infinity is timelike the cosmological constant must be negative. Hence the metric is locally anti-de Sitter, and since it is asymptotically simple it must be anti-de Sitter spacetime.
\end{proof}

The above proof rests on two lemmata. The first is an elementary polarization lemma. It holds in all dimensions $2$ and greater, so we will state (and prove) it without a dimension restriction.

\begin{lemma}[Polarization lemma]\label{lemma3.2}
Let $T$ be a symmetric $(0,2)$-tensor (at a point) such that $T(\ell,\ell)=0$ for every future-null vector (at that point). Then $T$ is proportional to $g$.
\end{lemma}

We prove this lemma in Appendix \ref{appendixB}. It applies above to show that $\ric$ is proportional to the spacetime metric. If we knew that $C(\ell,u,\ell,v)=0$ for arbitrary vectors $u,v\in T_pM$, we could apply the lemma to the $(0,2)$-tensors $C(\cdot,u,\cdot,v)$ as well. However, we only know that $C(\ell,u,\ell,v)=0$ for $u,v\in T_p^{\perp}M$, the orthogonal complement of $\ell$ at each $p$. Nonetheless, this suffices, as the following lemma shows. By an algebraic Weyl tensor on a vector space with Lorentzian inner product, we will mean a $(0,4)$-tensor $C$ such that $C_{abcd}=-C_{bacd}=-C_{abdc}$, $C_{abcd}=C_{cdab}$, $C_{abcd}+C_{bcad}+C_{cabd}=0$, and $C$ is tracefree in the sense that $g^{ac}C_{abcd}=0$.

\begin{lemma}\label{lemma3.3}
Let $(V, g)$ be an oriented Lorentzian vector space of dimension $4$, and
let $C$ be an algebraic Weyl tensor on $V$.
For every null vector $\ell$, assume that $C(X,\ell,Y,\ell)=0$
for all $X,Y\in\ell^\perp$ (the orthogonal complement of $\ell$).
Then $C = 0$.
\end{lemma}

The proof of Lemma \ref{lemma3.3}, like that of Lemma \ref{lemma3.2}, is purely algebraic and is given in Appendix \ref{appendixB}. But, simply put, $\ell$ will be a principal null direction of the Weyl tensor and there can be at most four distinct such directions at each point (indeed, by Goldberg-Sachs, $\ell$ will be a \emph{repeated} principal null direction, and there are at most two of those) unless the Weyl tensor vanishes. Thus, if every null direction is a principal null direction, the Weyl tensor must vanish.

\begin{remark}\label{remark3.4} Proposition \ref{proposition3.1} also holds if spacetime is 3-dimensional. In that case, $C$ and $\sigma$ vanish identically so Lemma \ref{lemma3.3} is not needed, while Lemma \ref{lemma3.2} is still required.
\end{remark}

\subsection{The answer assuming stasis}\label{subsection3.2}
In this subsection we do not assume the null energy condition, nor do we assume a dimensional restriction. Instead we assume that $(M,g)$ admits a complete, hypersurface-orthogonal, timelike (nonvanishing) Killing vector field which we denote by $\frac{\partial}{\partial t}$.

\begin{proposition}[Static rigidity]\label{proposition3.5}
Let $(M,g)$ be a conformally compactifiable, globally static spacetime. Assume that for each $p\in {\mathscr{I}}$, every null geodesic from $p$ passes through its spacetime antipode $q\in {\mathscr{I}}$. Then $(M,g)$ is conformally isometric to anti-de Sitter spacetime.
\end{proposition}

\begin{proof}
On the asymptotic end ($\Sigma$ outside a compact set), we have $\omega=0$. If $\Sigma$ is simply connected (or if $\omega$ is analytic), we can take $\omega=0$ everywhere and then $g=-N^2dt^2+h$. We will deal with this case first.

We identify spacetime points along integral curves of the static Killing vector field. This defines a map $\varphi:M\to \Sigma$ from spacetime $M$ to the quotient manifold $\Sigma$. We can regard $h$ as the pullback of a Riemannian metric on $\Sigma$. If we restrict $h$ and $\varphi$ to any constant-$t$ hypersurface in $M$, then $\varphi$ is an isometry, so we will use $h$ to denote both the Riemannian metric on $\Sigma$ and its pullback (which has kernel spanned by $\frac{\partial}{\partial t}$) to $M$. Under $\varphi$, it is well-known that null geodesics on $M$ project to curves on $\Sigma$ that are in fact constant-speed geodesics of the Fermat metric ${\tilde h}=h/N^2$; see the Appendix and set $\omega=0$ in the calculation there. Also, the converse holds: every geodesic in $(\Sigma,{\tilde h})$ lifts along integral curves of $\frac{\partial}{\partial t}$ to a null geodesic in $(M,g)$.

The map $\varphi$ extends to ${\mathscr{I}}$, whose image points will then be points of $\partial \Sigma$.

Now future-null geodesics in the spacetime obey
\begin{equation}
\label{eq3.3}
dt(\ell)=\sqrt{h(\ell,\ell)/N^2}=\sqrt{{\tilde h}(\ell,\ell)}
\end{equation}
along each future-null geodesic $\gamma$ with tangent field $\ell$ (in this expression ${\tilde h}$ is actually the pullback of ${\tilde h}$ from $\Sigma$ to $M$), and we may integrate along $\gamma$ from $p\in {\mathscr{I}}$ to its spacetime antipode $q\in {\mathscr{I}}$ to obtain
\begin{equation}
\label{eq3.4}
\begin{split}
T_q-T_p=&\, \int_{\gamma} \sqrt{{\tilde h}(\ell,\ell)}\\
=&\,  L_{\tilde h}(\gamma) .
\end{split}
\end{equation}
We used Lemma \ref{lemma2.2} here to guarantee that the endpoints on ${\mathscr{I}}$ are in fact spacetime antipodes. Here $T_r$ is the $t$-coordinate of a point $r\in {\mathscr{I}}$. Let ${\tilde p}$ denote the image point of $p$ in $\partial \Sigma$ under the identification above, and define ${\tilde q}$ likewise. Since all null geodesics from $p$ arrive at its spacetime antipode $q$, \eqref{eq3.4} says that every geodesic in $\Sigma$ from ${\tilde p}$ to ${\tilde q}$ has the same ${\tilde h}$-arclength.

Now the difference $T_q-T_p$ is the difference in the time coordinates of spacetime antipodes. Hence we can evaluate the integral on the right in \eqref{eq3.4} on a null geodesic along ${\mathscr{I}}$. Then $h=h_0=g({\mathbb S}^{n-1},{\rm can})$ and so $T_q-T_p=\pi$. In particular, it is independent of $p$. Hence every geodesic from $\partial \Sigma$ to $\partial \Sigma$ has ${\tilde h}$-arclength $\pi$. But by the main theorem of \cite{Bangert}, then $(\Sigma,{\tilde h})$ is isometric to a hemisphere of a round $n$-sphere. Hence $-dt^2+{\tilde h}=-dt^2+h/N^2$ is the Einstein static universe metric on the product of a line (the time coordinate) with an open hemisphere. It is well-known that there is a global conformal isometry mapping anti-de Sitter spacetime to this domain in the Einstein static universe.

Finally, we deal with the case where $\Sigma$ is not simply connected. By assumption, $\Sigma$ has a conformal infinity which is a round sphere, so there is a simply connected collar of conformal infinity ${\mathcal N}\subset \Sigma$. Arguing as above, we can set $\omega=df=0$ in the collar.

Now consider a geodesic $\Gamma$ which begins at point $p\in{\mathcal N}$ and ends at $q\in{\mathcal N}$ but which is not homotopic to a curve in ${\mathcal N}$. We can define a tubular neighbourhood ${\mathcal T}\in\Sigma$ such that $\Gamma\subset{\mathcal T}$. Then ${\mathcal N}\cap{\mathcal T}$ is open and may have two (or more) disjoint connected components, say ${\mathcal N}^-$ whose closure contains $p$ and ${\mathcal N}^+$ whose closure contains $q$. We can write that $\omega=dF$ on ${\mathcal T}$. Since $\omega=0$ on ${\mathcal N}$, then $F$ must be constant on each of ${\mathcal N}^-$ and ${\mathcal N}^+$. If $F$ were to equal the same constant on both regions, say $C$, we could simply redefine $\omega\vert_{\mathcal T}=d(F-C)$ and then extend $f$ to all of ${\mathcal N}\cup{\mathcal T}$. Then we would have $\omega=df$, on all of ${\mathcal N}\cup{\mathcal T}$. Continuing in this manner until $\Sigma$ is covered, we would obtain $\omega=df$ globally. But then, by a previous argument, we can set $\omega=0$ globally and apply the above proof.

But if $F$ is given by distinct constants on these two regions, say by $F=C^-$ on ${\mathcal N}^-$ and by $F=C^+\neq C^-$ on ${\mathcal N}^+$, then we can set $ds=0$ in equation \eqref{eqA.1} and integrate using the fundamental theorem for line integrals, we obtain
\begin{equation}
\label{eq3.5}
\begin{split}
&\, T_q-T_p+\int_{\Gamma}\omega =\int_{\Gamma}\sqrt{{\tilde h} \left ({\dot \Gamma}, {\dot \Gamma}\right ) }dt= L_{\tilde h}(\Gamma)\\
\implies &\, \pi +\int_{\Gamma}\omega =L_{\tilde h}(\Gamma).
\end{split}
\end{equation}
But if we reverse $\Gamma$ (in a mild abose of notation call the reversed curve $-\Gamma$), the integral curve of the resulting vector field is also a geodesic joining antipodes. The same calculation that yielded \eqref{eq3.5} now yields
\begin{equation}
\label{eq3.6}
\pi+\int_{-\Gamma}\omega =\pi -\int_{\Gamma}\omega =L_{\tilde h}(-\Gamma)= L_{\tilde h}(\Gamma).
\end{equation}
Then $\int_{\Gamma}\omega=0$ so $C^+=C^-$, $L_{\tilde h}{\Gamma}=\pi$, and since this holds for all geodesics connecting antipodes then by the main theorem of \cite{Bangert} ${\tilde h}$ is the metric on a hemisphere of a round sphere, $h$ is conformal to the round hemisphere metric, and the spacetime metric $g$ is conformal to the anti-de Sitter spacetime metric.
\end{proof}

For a geodesic $\gamma$ with endpoints $p$ and $q$ on the boundary of $\Sigma$, the quantity
\begin{equation}
\label{eq3.7}
A(\gamma):=L_{\tilde h}(\gamma)-\int_{\gamma} \omega \equiv \int_{\gamma} \left( |\dot{\gamma}|_g - \omega(\dot{\gamma}) \right) dt
\end{equation}
is known as the boundary case of \emph{Ma\~n\'e's action potential} (see for example \cite{DPSU}). Extremizing curves are known as \emph{magnetic pre-geodesics}. Constant speed extremizers are \emph{magnetic geodesics}. The study of magnetic geodesics will allow us to generalize from the static to the stationary case in the next section.

\section{The stationary case} \label{section4}
\setcounter{equation}{0}

\subsection{Magnetic hemispheres}\label{section4.1}Let $M$ be a compact connected $n$-manifold ($n\geq 2$) with $\partial M \cong {\mathbb S}^{n-1}$ (e.g. $M$ is a closed hemisphere of ${\mathbb S}^{n}$).
Let $g$ be a Riemannian metric on $M$ such that the boundary is totally geodesic and the metric on $\partial M$ agrees with the standard round metric. Let $\Omega \in \Omega^2(M)$ be a 2-form such that
\begin{itemize}
\item [i)] $\Omega=d\omega$ and
\item [ii)] the pullback of $\omega$ to $\partial M$ under inclusion vanishes.
\end{itemize}
The pair $(g,\Omega)$ defines a magnetic flow on the unit tangent bundle $SM$ via the ODE
\begin{equation}
\label{eq4.1}
\nabla_{\dot{\gamma}} \dot{\gamma} = Y(\dot{\gamma}),
\end{equation}
where $Y$ is defined by
\begin{equation}
\label{eq4.2}
\langle Y(X), Z \rangle_g = \Omega(X, Z)=d\omega(X,Z).
\end{equation}
The magnetic geodesic $\gamma_{x,v}$ is the unique solution to this ODE with initial condition $(x,v)\in SM$.

Given a boundary point $x\in \partial M$,  we let $S_x^+ M$ denote the open hemisphere of inward-pointing unit tangent vectors at $x$. We are interested in rigidity results under the following assumption: for any $x\in \partial M$ and $v\in S_x^+ M$, the magnetic geodesic $\gamma_{x,v}$ returns to the boundary $\partial M$ at an ``exit point'' $y(v)$ with constant Ma\~n\'e action $A(\gamma_{x,v})$.
\begin{equation}\label{eq:A=pi}
    A(\gamma_{x,v}) =\int_{\gamma_{x,v}} \left( |\dot{\gamma}_{x,v}|_g - \omega(\dot{\gamma}_{x,v}) \right) dt = \pi.
\end{equation}

We note some terminology used for magnetic geodesics that may be unfamiliar to general relativists. A magnetic geodesic (in the interior) joining two boundary points is a \emph{magnetic chord}. The initial data $(x,v)\in SM$, $x\in \partial M$, for a magnetic chord will also be called the \emph{entry data} or \emph{entry state}. When the magnetic chord returns to the boundary after travelling through the bulk, the point of return is the \emph{exit point} and the corresponding element of $SM$ is the \emph{exit data} or \emph{exit state}. Entry states are labelled with a `$+$' sign and exit states with a `$-$' sign. For $x \in \partial M$, a vector $v\in T_x \partial M\subset T_xM$ will be called ``glancing'', as will any trajectory through $x$ whose tangent vector at $x$ is parallel to a glancing vector $v$.

\subsection{The rigidity result}\label{section4.2}
The Lorentzian problem that we seek to address assumes that all null geodesics from $p\in{\mathscr{I}}$ refocus at some $q\in{\mathscr{I}}$ which we call the spacetime antipode on ${\mathscr{I}}$ of $p$, and from this it follows both that the related magnetic geodesics in the optical manifold refocus to join antipodes (in the usual sense, not the spacetime sense) on the boundary and that the Ma\~n\'e action of every inextendible magnetic geodesic is $\pi$. In what follows we will solve a different problem, which is a problem concerning magnetic geodesics and whose statement is more general than is needed for the Lorentzian problem. We will only assume that all magnetic chords have Ma\~n\'e action $\pi$; in particular we do not need to assume the magnetic flow to be non-trapping nor do we need to assume that the magnetic chords connect antipodes. Rather, this will follow from the other assumptions.

Our main result here is the following magnetic analogue of the result of Bangert \cite{Bangert} used in the previous section.

\begin{theorem} \label{theorem4.1}
Let $(M,g)$ be a compact connected Riemannian manifold-with-boundary, whose boundary is a round $(n-1)$-sphere (i.e., $g$ restricts to the canonical metric on $\partial M\simeq {\mathbb S}^{n-1}$). Suppose that
\begin{itemize}
\item [i)] the inclusion $\iota:\partial M\to M$ embeds $\partial M$ as a totally geodesic hypersurface in $M$,
\item [ii)] there is a 1-form $\omega$ defined on $M$ whose pullback obeys $\iota^*\omega=0$,
\item [iii)] for each magnetic chord $\gamma$ solving \eqref{eq4.1} and \eqref{eq4.2} (with $\omega$ as in (ii)), the Ma\~n\'e action obeys $A(\gamma)=\pi$, and
\item [iv)] $\max_{x\in M}|\omega_{x}|_{g}<1$.
\end{itemize}
Then $d\omega =: \Omega \equiv 0$.\label{thm:MagBangert}
\end{theorem}

\begin{remark}\label{remark4.2}
The bound on $\max_{x\in M}|\omega_{x}|_{g}$ (condition (iv) above) can be removed if instead the magnetic flow defined by \eqref{eq4.1} is assumed to be non-trapping and the Riemannian metric $g$ satisfies $d_{g}(x,-x)=\pi$ for each antipodal pair $x, -x\in \partial M$.
\end{remark}

This theorem is more general than is needed to address Question \ref{question1.2} for stationary asymptotically AdS spacetimes. In the Lorentzian setting, we assume that null geodesics join spacetime antipodes on ${\mathscr{I}}$, from which it will follow that magnetic chords in the optical manifold $\Sigma$ always join antipodes on $\partial \Sigma$. From this fact we easily deduce the constancy of the Ma\~n\'e action. But Theorem \ref{theorem4.1} holds whenever the Ma\~n\'e action for each chord is the same, even if it were not assumed that the chord joined antipodes.

Our main results follow as corollaries.
\begin{corollary}\label{corollary4.3}
When the assumptions of Theorem \ref{theorem4.1} hold, $g$ is isometric to the standard unit hemisphere.
\end{corollary}

\begin{proof}
Once we know that $\Omega$ vanishes, the result follows immediately from the main result in \cite{Bangert}.
\end{proof}

\begin{theorem}[Stationary rigidity] \label{theorem4.4}
Consider an asymptotically AdS, globally stationary spacetime with metric $g$ of the form \eqref{eq2.2} with
stationary Killing vector field $\frac{\partial}{\partial t}$ such that the $t=const$ level sets are spacelike hypersurfaces  and $ds^2/N^2$ induces on $\mathscr{I}$ the metric $-dt^2\oplus g(S^{n-1},\mathrm{can})$ of Definition 2.1.(v).
Assume that for each $p\in {\mathscr{I}}$, every null geodesic from $p$ passes through its spacetime antipode $q\in {\mathscr{I}}$. Then the spacetime is conformally isometric to anti-de Sitter spacetime.
\end{theorem}

\begin{proof}
Let $(\Sigma,{\tilde h})$ denote the optical manifold and metric as described by Definitions \ref{definition2.4} and \ref{definition2.5}, obtained by identifying points in spacetime if they lie along the same integral curve of the Killing vector field $\frac{\partial}{\partial t}$ and dividing the quotient space metric $h$ by $N^2$, the square of the lapse function.

If the $t=const$ slices of spacetime are spacelike then $h-N^2\omega^2=N^2\left ( {\tilde h}-\omega^2\right )$ is a Riemannian metric. Hence $\left \vert \omega \right \vert_{\tilde h} <1$, where ${\tilde h}=h/N^2$ is the optical metric on the quotient Riemannian manifold (identifying points, as described in the Appendix, along integral curves of the Killing vector field $\frac{\partial}{\partial t}$).

Let ${\dot \gamma}$ denote the tangent vector of an affinely parametrized geodesic $\gamma$ of the metric \eqref{eq2.2}. Then ${\dot \gamma}$ obeys
\begin{equation}
\label{eq4.4}
dt\left ( {\dot \gamma}\right )+\omega\left ( {\dot \gamma}\right ) =\pm \sqrt{{\tilde h}\left ( {\dot \gamma},{\dot \gamma} \right )}.
\end{equation}
Say that $\gamma$ departs ${\mathscr{I}}$ at point $p$ with time coordinate $t=T_p$ and returns to ${\mathscr{I}}$ at point $q$ with time coordinate $t=T_q$. By assumption, the time separation of antipodes is $T_q-T_p=\pi$
so
\begin{equation}
\label{eq4.5}
\pi=\int_{\gamma} \left ( \left \vert {\dot \gamma}\right \vert_{\tilde h} -\omega\left ( {\dot \gamma}\right )\right )dt \equiv A(\gamma),
\end{equation}
where, consistent with our usual abuse of notation, the integral is actually taken over $\varphi\circ\gamma$, the image of $\gamma$ under the map $\varphi:M\to\Sigma$.
Thus magnetic chords in the optical manifold all have Ma\~n\'e action $\pi$ and the assumptions of Theorem \ref{theorem4.1} are fulfilled.
By Corollary \ref{corollary4.3}, the optical manifold is a round hemisphere, so spacetime is conformal to a domain in the Einstein static universe $\left ({\mathcal R}\times {\mathbb S}^n, -dt^2+g({\mathbb S}^n,{\rm can})\right )$.
\end{proof}

\subsection{Preliminaries}\label{section4.3}
Given the inclusion $\iota:\partial M\to M$, then condition $\iota^*\omega=0$ immediately implies that $\iota^{*}(d\omega)\equiv \iota^{*}\Omega=0$. Taking this together with the constancy of the Ma\~n\'e action on magnetic chords (equation \eqref{eq:A=pi}), we obtain more:

\begin{lemma} \label{lemma4.5}
If the Ma\~n\'e action of every magnetic chord $\gamma$ obeys $A(\gamma)=\pi$ and if $\iota^*\omega=0$, then $d\omega\vert_x=\Omega_{x}=0$ for all $x\in\partial M$.\label{lemma:Y=0}
\end{lemma}

\begin{proof}
Let $x \in \partial M$ and let $\nu$ be the inward-pointing unit normal vector at $x$. Because the pullback of $\Omega$ to the boundary vanishes, $\Omega_x(u,v) = 0$
for any two tangent vectors $u,v \in T_x \partial M$. Therefore, it suffices to show that the linear functional $v \mapsto \Omega_x(v, \nu)$ on $T_x \partial M$ is trivial.

Suppose, for the sake of contradiction, that $\Omega_x$ is not identically zero. By linearity, there must exist a unit ``glancing'' vector $v \in S_x \partial M$ such that $\Omega_x(v, \nu) = -c < 0$. Let $\{v_n\} \subset S_x^+ M$ be a sequence of strictly inward-pointing unit vectors converging to $v$
and let $\gamma_n$ be the associated sequence of magnetic chords.
In boundary normal coordinates $(x', x^\perp)$, the normal velocity is $v_n^\perp = \dot{x}^\perp(0) > 0$. Because $\partial M$ is totally geodesic, its second fundamental form vanishes. Thus, the normal acceleration
of $\gamma_n$
at $t=0$ is given by $\ddot{x}_n^\perp(0) = g(Y(v_n),\nu(x_n)) =: -c_{n}$. Since $c_{n}\to c>0$ the normal acceleration at $t=0$ is strictly negative and bounded away from zero for large $n$.
By the compactness of $SM$, we can Taylor expand the normal coordinate of the trajectory $\gamma_n$ associated with $v_n$ with a uniform error term
\begin{equation}
    x_n^\perp(t) = v_n^\perp t - \frac{1}{2} c_{n} t^2 + O(t^3) .
\end{equation}
Because the initial normal velocities obey $v_n^\perp > 0$ but the accelerations obey $-c_{n}<0$,
the trajectories are forced back into the boundary at exit times $T_n$ corresponding to the positive root of $x^\perp(T_n) = 0$. Discarding higher-order terms, these ``exit times'' are asymptotic to:
\begin{equation}
    T_n \approx \frac{2v_n^\perp}{c_{n}}
\end{equation}
As $v_n \to v$, the normal velocity obeys $v_n^\perp \to 0$, which implies that the exit times obey $T_n \to 0$.

Now we evaluate the Ma\~n\'e action of these ``short chords''. Because $SM$ is compact, $1 - \omega_{x}(v)$ is bounded above by some constant $K > 0$. Thus, the action is bounded by
\begin{equation}
    A(\gamma_n) = \int_0^{T_n} (1 - \omega(\dot{\gamma}_n)) dt \le K T_n .
\end{equation}
Because $T_n \to 0$, this contradicts $A(\gamma_n) = \pi$ for all $n$.
\end{proof}

\begin{lemma}[Trapped glancing trajectories] \label{lem:trapped}
If $\partial M$ is totally geodesic and $\Omega_x=0$ for all $x\in\partial M$, then for any vector $v \in T_x \partial M$ with $|v|_g = 1$ the magnetic geodesic $\gamma_{x,v}$ is a geodesic and remains within $\partial M$.\label{lemma4.6}
\end{lemma}

\begin{proof}
Because $\partial M$ is totally geodesic, any solution of the geodesic equation that is initially tangent to $\partial M$ remains in $\partial M$. By hypothesis, $\Omega_x=0$ for $x\in\partial M$, so the geodesic and magnetic geodesic equations agree on $\partial M$. Hence every geodesic initially tangent to $\partial M$ is a magnetic geodesic that remains in $\partial M$, and by the uniqueness of solutions of the magnetic geodesic equation then all magnetic geodesics initially tangent to $\partial M$ remain in $\partial M$.
\end{proof}

The next lemma is standard. We include a proof for completeness.

\begin{lemma}[The scattering relation] \label{lem:scattering}
Assume the assumptions of Lemma \ref{lemma4.6} hold and that the magnetic flow is non-trapping. The scattering relation $\mathcal S:=\mathcal{S}_{g,\Omega}: \partial_+ SM \to \partial_- SM$ which maps $(x,v)\in\partial_+SM$ to the exit position and velocity of $\gamma_{x,v}$, is a diffeomorphism.\label{lemma4.7}
\end{lemma}

\begin{proof}
Let $\phi^t: SM \to SM$ denote the magnetic flow. Because the flow is non-trapping, for every $(x,v) \in \partial_+ SM$, there exists a strictly positive first time $\tau(x,v) > 0$ such that the trajectory is incident on the boundary: $\phi^{\tau(x,v)}(x,v) \in \partial SM$. We write this ``exit state'' as $(y,w) = \phi^{\tau(x,v)}(x,v)$. The scattering relation is defined by $\mathcal{S}_{g,\Omega}(x,v) = (y,w)$.

To show that $\mathcal{S}_{g,\Omega}$ is smooth, we first prove that the exit time function $\tau: \partial_+ SM \to (0,\infty)$ is smooth. We must verify that the trajectory intersects the boundary transversally at the exit point.

Suppose, for the sake of contradiction, that the trajectory exits tangentially, meaning that $w \in T_y \partial M$. By Lemma \ref{lem:trapped}, the magnetic geodesic with initial condition $(y,w)$ must remain entirely within $\partial M$ for all time. By the uniqueness of solutions to the ODE, the entire backward trajectory must also be contained in $\partial M$. This contradicts the assumption that the initial state $(x,v)$ lies in $\partial_+ SM$, which requires $\langle v, \nu(x) \rangle_g > 0$ (i.e., $v$ is strictly inward-pointing). Therefore, the trajectory cannot exit tangentially. Since it is exiting, the velocity $w$ must be strictly outward-pointing, meaning $(y,w) \in \partial_- SM$ and $\langle w, \nu(y) \rangle_g < 0$.

Now, let $\rho: M \to \mathbb{R}$ be a smooth boundary defining function such that $\partial M = \rho^{-1}(0)$ and the inward unit normal is given by $\nabla \rho = \nu$ on $\partial M$. We define $F: \partial_+ SM \times (0,\infty) \to \mathbb{R}$ by
\begin{equation}
    F(x,v,t) = \rho(\pi \circ \phi^t(x,v)) ,
\end{equation}
where $\pi: SM \to M$ is the canonical projection. The exit time $\tau(x,v)$ is a root of this function: $F(x,v,\tau(x,v)) = 0$. We evaluate the partial derivative of $F$ with respect to time at the exit point to be
\begin{equation}
    \frac{\partial F}{\partial t}(x,v,\tau(x,v)) = d\rho(w) = \langle \nu(y), w \rangle_g < 0.
\end{equation}
Because this derivative is non-zero, the implicit function theorem guarantees that the root $\tau(x,v)$ depends smoothly on the initial conditions $(x,v)$. Consequently, the composition $\mathcal{S}_{g,\Omega}(x,v) = \phi^{\tau(x,v)}(x,v)$ is a smooth map from $\partial_+ SM$ to $\partial_- SM$.

Finally, we show that $\mathcal{S}_{g,\Omega}$ is a diffeomorphism by constructing a smooth inverse. Because the forward flow is non-trapping, the invertibility of the flow map $\phi^t$ ensures that integrating the ODE backward in time is also non-trapping. For any $(y,w) \in \partial_- SM$, we can define a backward entry time $\tau_-(y,w) > 0$. By the exact same transversality argument, because the entry velocity $v$ satisfies $\langle v, \nu(x) \rangle_g > 0$, the backward entry time function $\tau_-$ is smooth.

The inverse map $\mathcal{S}_{g,\Omega}^{-1}: \partial_- SM \to \partial_+ SM$ is given by $\mathcal{S}_{g,\Omega}^{-1}(y,w) = \phi^{-\tau_-(y,w)}(y,w)$, which is a composition of smooth functions. Since both $\mathcal{S}_{g,\Omega}$ and $\mathcal{S}_{g,\Omega}^{-1}$ are smooth bijections, $\mathcal{S}_{g,\Omega}$ is a diffeomorphism.
\end{proof}

\subsection{The free-time action}\label{section4.4}
Define the \emph{energy} of an element of $TM$ by
\begin{equation}
E(x,v):=\frac{1}{2}|v|_{g}^{2}.
\end{equation}
The energy of a curve $\gamma$ at $x=\gamma(t)$ is then $E(\gamma(t),\gamma'(t))$. Unit speed curves have energy $E=\frac12$ at every point. When dealing with variational properties of unit-speed magnetic geodesics it is convenient to work with the so-called {\it free-time action functional}. The free-time action of a curve $\gamma:[0,T]\to M$ is defined as
\begin{equation}
{\mathcal A}(\gamma):=\int_{0}^{T}\left (L(\gamma,\dot{\gamma})+\frac{1}{2}\right)\,dt,
\end{equation}
where
\begin{equation}
L(x,v)= \frac{1}{2}|v|_g^2 - \omega_x(v).
\end{equation}
If $\gamma$ has unit speed its free-time action equals its Ma\~n\'e action.

It is convenient to introduce the \emph{outflux} and \emph{influx boundaries}
\begin{equation}\label{eq4.14}
\partial_{\pm} SM := \{ (x, v) \in SM \mid x \in \partial M:\, \pm\langle v, \nu(x) \rangle_g > 0\},
\end{equation}
where $\nu(x)$ is the inward-pointing unit normal. We define the open unit cotangent ball bundle
\begin{equation}\label{eq4.15}
B^*\partial M = \{ (x, p) \in T^*\partial M:\, |p|_g < 1 \}
\end{equation}
and smooth diffeomorphisms $\psi_\pm : \partial_\pm SM \to B^*\partial M$ by
\begin{equation}
\label{eq4.16}
\psi_\pm(x, v):= \left (x, \left ( v^{\rm tan}\right )^\flat\right ) \equiv \left ( v_x^{\rm tan}\right )^\flat.
\end{equation}

\begin{lemma}[First variation of the free-time action] \label{lem:variation}
Let $s \mapsto (x(s), v(s))$ be a smooth 1-parameter family of entry states in $\partial_+ SM$, exiting at $y(s)$ with velocity $w(s)$ at time $T(s)$. Let $X = x'(0)$ and $Y = y'(0)$. Let
\begin{equation}
\label{eq4,17}
\begin{split}
P_{\rm in} :=&\, \psi_+(x(0),v(0))\equiv \left ( x(0),\left ( v^{\rm tan}(0)\right )^{\flat}\right ) \equiv \left ( v_{x(0)}^{\rm tan}(0)\right )^{\flat},\\
P_{\rm out} :=&\, \psi_-(y(0), w(0))\equiv \left ( y(0),\left ( w^{\rm tan}(0)\right )^{\flat}\right ) \equiv \left ( w_{y(0)}^{\rm tan}(0)\right )^{\flat}.
\end{split}
\end{equation}
Assume as well that $\omega\vert_{\partial M}=0$. Then the derivative of the free-time action of $\gamma_{x(s),v(s)}$ is
\begin{equation}
    \left. \frac{d}{ds} {\mathcal A}(\gamma_{x(s),v(s)}) \right|_{s=0} = P_{\rm out}(Y) - P_{\rm in}(X) .
\end{equation}
\end{lemma}

\begin{proof} We abbreviate $\gamma_{s}=\gamma_{x(s),v(s)}$ and ${\mathcal A}(s):={\mathcal A}(\gamma_{x(s),v(s)})$. Then
\begin{equation}
{\mathcal A}(s) =\int_0^{T(s)} \left( L(\gamma_s, \dot{\gamma}_s) + \frac{1}{2} \right) dt.
\end{equation}
Define the \emph{momentum} by
\begin{equation}
p = \frac{\partial L}{\partial v} = v^\flat - \omega .
\end{equation}
Let $V(t) = \frac{\partial \gamma_s}{\partial s}|_{s=0}$ be the variation vector field along $\gamma_0$. Differentiating ${\mathcal A}(s)$ yields:
\begin{equation}
    {\mathcal A}'(0) = \left( L_{\rm out} + \frac{1}{2} \right) T'(0) + \int_0^{T(0)} \left( \frac{\partial L}{\partial x}(V) + \frac{\partial L}{\partial v} (\dot{V}) \right) dt .
\end{equation}
Because $\gamma_0$ satisfies the Euler-Lagrange equations, integration by parts reduces the integral to the boundary terms $[p (V)]_0^{T_0}$.

Define the \emph{exit curve} (i.e., the curve of exit points) to be
\begin{equation}
y(s) = \gamma_s(T(s)).
\end{equation}
Hence the base variation at the exit is $Y = y'(0) = V(T_0) + \dot{\gamma}_0(T_0) T'(0)$, which gives $V(T_0) = Y - \dot{\gamma}_0(T_0) T'(0)$. The initial base variation is $V(0) = X$.

Substituting these into the boundary terms yields
\begin{equation}
    {\mathcal A}'(0) = \left(  L_{\rm out} + \frac{1}{2} -p_{\rm out}(\dot{\gamma}_{\rm out}) \right) T'(0) +p_{\rm out}(Y) - p_{\rm in}(X) .
\end{equation}
The coefficient of $T'(0)$ is exactly $1/2-E$. Because our trajectories lie on the energy level $E = 1/2$, this time-variation term vanishes.

Finally, because $\omega|_{T\partial M} = 0$ and the base variations $X$ and $Y$ obey $X, Y \in T\partial M$, the magnetic term in the momentum evaluates to zero. The evaluations reduce to the purely Riemannian tangential momenta $p_{\rm out}(Y) = \langle \dot{\gamma}_{\rm out}, Y \rangle_g = P_{\rm out}(Y)$ and $p_{\rm in}(X) = \langle \dot{\gamma}_{\rm in}, X \rangle_g = P_{\rm in}(X)$.
\end{proof}

\subsection{Auxiliary lemmas}\label{section4.5}
Our first objective is to show that under the non-trapping condition and assuming that every magnetic chord has Ma\~n\'e action $\pi$, the scattering relation ${\mathcal S}$ is identical to that of the standard hemisphere. For this we use a symplectic argument.

\begin{remark} {\rm A reader who prefers a shorter proof of the main rigidity result and is willing to assume the non-trapping property and antipodal focusing may skip Lemmas \ref{lem:cone} and \ref{lem:non_trapping} below and proceed directly to the second paragraph of the proof of Proposition \ref{prop:rs}.}
\end{remark}

Recall that $S_x^+M$ denotes the open hemisphere of inward-pointing unit tangent vectors at $x\in\partial M$ and that $B^*\partial M$ denotes the open unit cotangent ball bundle. For a fixed $x \in \partial M$, let $\Phi_x : S_x^+ M \to B^*\partial M$ be the map $\Phi_x(v) := (y(v), P_{\rm out}(v))$.

\begin{lemma} \label{lem:cone}
Assume that each magnetic chord has Ma\~n\'e action $\pi$, that the magnetic flow is non-trapping, and that $\omega\vert_{\partial M}=0$.
The image $\Lambda_x = \Phi_x(S_x^+ M)$ is a smooth conical Lagrangian submanifold in $T^*\partial M$. (By conical we mean a union of open radial segments in $B^*\partial M$, i.e., sets of the form $\{ t\,p : t \in (0,\epsilon) \}$ where $y\in \partial M$ and $0\neq p\in T_{y}^*\partial M$.)\label{lemma4.10}
\end{lemma}

Before proceeding with the proof, we first parse this lemma and its proof. As we vary the initial tangent vector $v$ at the fixed ``entry point'' $x\in\partial M$, the lemma says that the exit data $(y(v),w(v))$ define a smooth submanifold $\Lambda_x$ of dimension $n-1$ in $T^*\partial M$ parametrized (partly) by a coordinate $t\in (0,1)$ (called ``radial'' here) which is in fact the magnitude of the projection of the (unit) exit cotangent vector into $T^*\partial M$. Though we do not yet know that this submanifold is a fibre above a single exit point (which we will show in a subsequent lemma), it is a Lagrangian submanifold.

The first step is to consider a geodesic variation issuing from a fixed $x\in \partial M$. The Euler-Lagrange equations (the magnetic geodesic equations) for the free-time action guarantee that the bulk term in the variation vanishes, and since the variation issues from a single point the boundary term is entirely due to the ``exit'' point. This term must therefore vanish as well, since the action is constant along the variation. This implies that the projection of the exit velocity into the tangent space $T_y\partial M$ at the exit point $y$ is orthogonal in $T_y\partial M$ to the variation vector field at the exit point $y$ including the possibility that the variation vector field vanishes at $y$. This is true for every geodesic in the variation, not just a base geodesic, since the action is constant (not merely stationary at a base point) with respect to the variation. That means that $v^{
\flat}(y'(0))$ vanishes for all variations $y'(0)$. On $\partial M$, we have $p=v^{\flat}$ and in coordinates $y'(0)=\frac{dq}{ds}$ so $\lambda=pdq$ vanishes at exit points. By simple manipulations this allows us to argue that the submanifold is Lagrangian and is conical in the above sense. Single fibres in the cotangent bundle have this structure. In subsequent lemmas we will prove that the variation vector field does vanish at $y$, so that $\Lambda_x$ will be confined to the fibre above $y$.

\begin{proof}
Recall the scattering relation $\mathcal S:=\mathcal{S}_{g,\Omega}: \partial_+ SM \to \partial_- SM$ defined in Lemma \ref{lemma4.7}, which maps $(x,v)\in\partial_+SM$ to the exit position and velocity of $\gamma_{x,v}$. Then $\Phi_x= \psi_- \circ \mathcal{S} \circ \iota_x$, where $\iota_x : S_x^+ M \hookrightarrow \partial_+ SM$ is the natural fibre inclusion. Because $\iota_x$ is a smooth embedding and both $\mathcal{S}$ and $\psi_-$ are diffeomorphisms, $\Phi_x$ is a smooth embedding. Therefore, its differential $d\Phi_x$ has full rank $n-1$ everywhere, making $\Lambda_x$ a smooth, embedded submanifold of dimension $n-1$.

Let $s \mapsto v(s)$ be a 1-parameter family of velocities in $S_x^+ M$. These initial velocities serve as initial data for a smooth one-parameter family of magnetic geodesics $\gamma_{x,v(s)}$. Define $A(s):=A(\gamma_{x,v(s)})$. By hypothesis, $A(s) \equiv \pi$, so $\frac{d}{ds} A(s) = 0$. Since $\omega\vert_{\partial M}=0$, Lemma \ref{lem:variation} forces $P_{\rm out}(Y)-P_{\rm in}(X) = 0$, but $P_{\rm in}(X) = 0$ since the initial endpoint $x(0)=x$ is fixed so then $P_{\rm out}(Y)=0$. This means the pullback of the canonical Liouville 1-form $\lambda := p\,dq$ vanishes: $\Phi_x^*\lambda = 0$. Consequently, the pullback of the symplectic form vanishes: $\Phi_x^*(-d\lambda) = 0$.

Because $\Lambda_x$ is an $(n-1)$-dimensional isotropic submanifold, it is Lagrangian. For any tangent vector $Z \in T\Lambda_x$, the radial Liouville vector field $V = p \frac{\partial}{\partial p}$ satisfies $-d\lambda(V, Z) = -\lambda(Z) = 0$. Because $\Lambda_x$ is Lagrangian, its tangent space is its own symplectic orthogonal, forcing $V \in T\Lambda_x$. Since $V$ is tangent to $\Lambda_x$ everywhere, $\Lambda_x$ is conical.
\end{proof}

Next we show that $\Phi_x$ extends continuously to the boundary. The idea is that the $t\nearrow 1$ limit in the proof above implies that the limiting covector lies in $T_{-x}^*\partial M$ and is tangent to a geodesic $\gamma:[0,T]\to \partial M$, $T<\infty$.

\begin{lemma} \label{lem:extension}
Say that the assumptions of Lemma \ref{lemma4.10} hold and that $\partial M$ is embedded as a totally geodesic hypersurface in $M$. The map $\Phi_x$ extends continuously to the closed hemisphere $\overline{S_x^+ M}$, mapping  $\partial S_x^+ M$ onto $S_{-x}^*\partial M$.
\label{lemma4.11}
\end{lemma}

\begin{proof}
Let $v \in \partial S_x^+ M$ be tangent to $\partial M$; i.e., a ``glancing'' vector. Let $\{v_n\} \subset S_x^+ M$ be a sequence of strictly inward-pointing vectors such that $v_n \to v$. Let $T_n$ be the boundary exit time of the corresponding magnetic geodesic $\gamma_n(t) = \gamma_{x, v_n}(t)$. We claim that the sequence $T_{n}$ is bounded.

Suppose for the sake of contradiction that the sequence $T_{n}$ is not bounded, so without loss of generality assume $T_{n}\to\infty$.

For each $n$ we define a probability measure $\mu_n$ on $SM$ by setting
\begin{equation}
    \int_{SM} f \, d\mu_n = \frac{1}{T_n} \int_0^{T_n} f(\dot{\gamma}_n(t)) \, dt.
\end{equation}
Since the space of probability measures on the compact metric space $SM$ is weak-* compact, we can extract a subsequence (which we still denote by $\mu_n$) that converges in the weak-* topology to a limit probability measure $\mu$. Because $T_n \to \infty$, a standard argument shows that the weak-* limit measure $\mu$ is invariant under the magnetic flow.

By the non-trapping hypothesis, every trajectory that enters the interior of $M$ exits in finite time. The only trapped trajectories are those confined to the totally geodesic boundary $\partial M$, where both the normal acceleration and $\Omega$ vanish. Because $\mu$ is an invariant probability measure, its support must be contained within the trapped set of the dynamics. Consequently, $\text{supp}(\mu) \subseteq S\partial M$.
We now evaluate the weak-* limit using the continuous function $L(x,v) = 1 - \omega_{x}(v)$ on $SM$. By definition of weak-* convergence we have
\begin{equation} \label{eq:weak_limit}
    \lim_{n \to \infty} \int_{SM} (1 - \omega) \, d\mu_n = \int_{SM} (1 - \omega) \, d\mu.
\end{equation}
We evaluate the left-hand side of \eqref{eq:weak_limit} using the constant Ma\~n\'e action hypothesis, $A(\gamma_n) = \pi$:
\begin{equation}\label{eq:aux}
    \int_{SM} (1 - \omega) \, d\mu_n = \frac{1}{T_n} \int_0^{T_n} \left( 1 - \omega(\dot{\gamma}_n(t)) \right) dt = \frac{A(\gamma_n)}{T_n} = \frac{\pi}{T_n} .
\end{equation}
Taking the limit as $T_n \to \infty$, we see that the left-hand side of \eqref{eq:aux} tends to $0$.
Next, we evaluate the right-hand side of \eqref{eq:weak_limit}. Because the support of $\mu$ is contained in $S(\partial M)$, we only evaluate the integrand on boundary tangent vectors. Since $\omega|_{T\partial M} \equiv 0$ we see that
\begin{equation}
    \int_{SM} (1 - \omega) \, d\mu = \int_{S(\partial M)} (1 - 0) \, d\mu = \int_{S(\partial M)} 1 \, d\mu = 1 ,
\end{equation}
thus giving a contradiction.

Now consider a subsequence of $T_{n}$ such that $T_{n_k} \to T^*$. Then
\begin{equation}
    \lim_{k \to \infty} \int_0^{T_{n_k}} \omega(\dot{\gamma}_{n_k}(t)) \, dt = \int_0^{T^*} \omega(\dot{\gamma}_v(t)) \, dt=0.
\end{equation}
Taking the limit of the action equation for the subsequence yields
\begin{equation}
    \lim_{k \to \infty} \left( T_{n_k} - \int_0^{T_{n_k}} \omega(\dot{\gamma}_{n_k}(t)) \, dt \right) = T^* - 0 = \pi .
\end{equation}
Thus, $T^* = \pi$. Because every convergent subsequence of the bounded sequence $\{T_n\}$ must converge to the same limit $\pi$, the sequence itself converges: $\lim_{n \to \infty} T_n = \pi$.

Because both the trajectories and their exit times converge, the final exit states converge: $\gamma_{x,v_n}(T_n) \to \gamma_{x,v}(\pi)$ and $\dot{\gamma}_{x,v_n}(T_n) \to \dot{\gamma}_{x,v}(\pi)$. Because Lemma \ref{lem:trapped} ensures $\gamma_{x,v}$ is a unit-speed geodesic in the round sphere $\partial M$, we have $\gamma_{x,v}(\pi)=-x$.
Since the initial glancing vectors $v \in \partial S_x^+ M$ sweep out all directions in the unit tangent sphere of $\partial M$ at $x$, the exit momenta sweep out the unit cotangent sphere at $-x$. Thus, $\Phi_x$ extends continuously to the closure, mapping $\partial S_x^+ M$ onto $S_{-x}^{*}\partial M$.
\end{proof}

\begin{proposition}[Rigidity of the scattering relation]\label{proposition4.12}
Assume the magnetic flow is non-trapping and each magnetic chord has Ma\~n\'e action $\pi$.
Then the endpoint map is constant, $y(v) \equiv -x$, and the scattering relation $\mathcal{S}_{g,\Omega}$ matches that of the standard hemisphere $\mathcal{S}_{{\rm can}}$.\label{prop:rs}
\end{proposition}

\begin{proof}
Every radial ray $R$ in the cone $\Lambda_x$ must eventually approach $|p| = 1$, reaching the boundary $\partial \Lambda_x$ within the closed ball bundle $\overline{B^*\partial M}$. By Lemma \ref{lem:extension}, this boundary is exactly $\partial \Lambda_x = \Phi_x(\partial S_x^+ M) = S_{-x}^*\partial M$.
Because a radial ray $R$ generated by the Liouville vector field lies entirely within a single cotangent fibre over its exit point $y$, it cannot intersect the target fibre $T_{-x}^*\partial M$ unless its base coordinate is identically $y = -x$. Therefore, $y(v) \equiv -x$ for all $v \in S_x^+ M$.

Because this holds for every boundary point $x \in \partial M$, the base mapping is the antipodal map $y = \mathfrak{A}(x) = -x$.
We now evaluate arbitrary 1-parameter families where both $x$ and $v$ vary. Because the action is identically $\pi$, Lemma \ref{lem:variation} forces $P_{\rm out}(Y) - P_{\rm in}(X) = 0$.
Since $y(s) = \mathfrak{A}(x(s))$, the base variations are coupled by the pushforward: $Y = d\mathfrak{A}(X)$.
Substituting this yields $P_{\rm out}(d\mathfrak{A}(X)) = P_{\rm in}(X)$, which means $d\mathfrak{A}^* P_{\rm out} = P_{\rm in}$. This forces the exit momentum to be the standard cotangent lift: $P_{\rm out} = (d\mathfrak{A}^*)^{-1} P_{\rm in}$.
Because both the base map and the momentum map are identical to those of the standard round hemisphere, we conclude that $\mathcal{S}_{g,\Omega} = \mathcal{S}_{{\rm can}}$ everywhere.
\end{proof}

Next we show that the non-trapping property can be derived from the assumption that magnetic geodesics between boundary points have constant action together with the bound
\begin{equation}
\max_{x\in M}|\omega_{x}|_{g}<1.
\end{equation}

\begin{lemma} \label{lem:non_trapping}
Assume the boundary $\partial M$ is totally geodesic and $\Omega_x = 0$ for all $x \in \partial M$. Suppose the magnetic potential satisfies the bound $\max_{x \in M} |\omega_x|_g < 1$. If every magnetic chord connecting boundary points has Ma\~n\'e action $\pi$, then the magnetic flow is non-trapping. \label{lemma4.13}
\end{lemma}

\begin{proof}
Let $SM^\circ = SM \setminus \partial SM$ denote the interior of the unit tangent bundle. We define a \textit{magnetic chord} as a trajectory $\gamma_{x,v}$ with initial state $(x,v) \in \partial_+ SM$ that reaches the boundary at a finite exit time $\tau(x,v) > 0$. Let $\mathcal{U} \subset SM^\circ$ be the set of all states that belong to a magnetic chord.

To prove the flow is non-trapping, we must show that every interior state lies on a chord, meaning $\mathcal{U} = SM^\circ$. Since $SM^\circ$ is connected, and assuming the existence of at least one chord so that $\mathcal{U} \neq \emptyset$, it suffices to prove that $\mathcal{U}$ is both open and closed in $SM^\circ$. The chord set is nonempty. Indeed, by the magnetic Santaló formula (see for example \cite[Equation VII.4.1, p 339]{Chavel}, and for the magnetic case see \cite[Appendix A.4]{DPSU}; the proof in the latter reference applies in the non-trapping case), or equivalently by the invariance of Liouville measure under the magnetic flow, almost every entry state has finite exit time. If a positive-measure set of entry states had infinite exit time, its flow-out for times $[0,T]$ would have volume growing linearly in $T$ inside the compact manifold $SM$, a contradiction. Hence at least one magnetic chord exists.

By the bound $\max_{x \in M} |\omega_x|_g < 1$ and the compactness of $M$, there exists a constant $\delta > 0$ such that for any state $(y,w) \in SM$
\begin{equation}
    |\omega_y(w)| \le 1 - \delta \implies 1 - \omega_y(w) \ge \delta > 0.
\end{equation}
For any magnetic chord $\gamma_{x,v}$, the Ma\~n\'e action is identically $\pi$. Since the trajectory is parameterized by unit speed ($|\dot{\gamma}_{x,v}|_g = 1$), its transit time $\tau(x,v)$ is bounded by
\begin{equation}
    \pi = A(\gamma_{x,v}) = \int_0^{\tau(x,v)} \left( 1 - \omega(\dot{\gamma}_{x,v}(t)) \right) dt \ge \int_0^{\tau(x,v)} \delta \, dt = \delta \tau(x,v).
\end{equation}
This establishes a uniform upper bound for the exit time of all magnetic geodesics: $\tau(x,v) \le \pi / \delta$.

Let $\phi^t: SM \to SM$ denote the magnetic flow, and let $w \in \mathcal{U}$. By definition, $w = \phi^{t_0}(x,v)$ for some $(x,v) \in \partial_+ SM$ and some interior evaluation time $t_0 \in (0, \tau(x,v))$. Because $\partial M$ is totally geodesic and $\Omega$ vanishes there, trajectories cannot be tangent to the boundary without remaining in it permanently (Lemma \ref{lem:trapped}). Since the chord enters and exits the interior, it must intersect the boundary transversally at both endpoints. By the smooth dependence of the flow on initial conditions, any state sufficiently close to $w$ in $SM^\circ$ will also belong to a trajectory that intersects the boundary transversally in finite time. Thus, $\mathcal{U}$ is open.

To show that $\mathcal{U}$ is closed, let $\{w_n\} \subset \mathcal{U}$ be a sequence  converging to a limit $w \in SM^\circ$. Each $w_n$ belongs to a chord generated by some $(x_n, v_n) \in \partial_+ SM$, evaluated at time $t_n \in (0, T_n)$, where $T_n = \tau(x_n, v_n)$. Thus, $w_n = \phi^{t_n}(x_n, v_n)$.

Because $SM$ is compact and the exit times $T_n$ are uniformly bounded by $\pi / \delta$, we can extract a convergent subsequence (which we still denote by index $n$) such that the initial states $(x_n, v_n) \to (x_\infty, v_\infty) \in \overline{\partial_+ SM}$, the exit times $T_n \to T_\infty \le \pi / \delta$, and the evaluation times $t_n \to t_\infty$. By the continuity of the flow, the limit state satisfies $\phi^{t_\infty}(x_\infty, v_\infty) = w$.

We must verify that the limit trajectory generated by $(x_\infty, v_\infty)$ is a valid chord. The initial state $(x_\infty, v_\infty)$ belongs to the closure $\overline{\partial_+ SM} = \partial_+ SM \cup S(\partial M)$. Suppose, for the sake of contradiction, that $(x_\infty, v_\infty) \in S(\partial M)$, meaning it is a unit-speed glancing vector. By Lemma \ref{lem:trapped}, the unique solution to the magnetic ODE for a glancing initial vector is a curve permanently trapped within $\partial M$. However, this contradicts the fact that the trajectory passes through $w \in SM^\circ$.

Therefore, $(x_\infty, v_\infty) \in \partial_+ SM$ is strictly inward-pointing. Furthermore, because the trajectory reaches $w \in SM^\circ$, the evaluation time $t_\infty$ (and thus the total exit time $T_\infty$) must be strictly greater than zero. Because $T_\infty \le \pi/\delta$ is finite, the limit trajectory is a valid magnetic chord that exits the manifold. Thus, $w \in \mathcal{U}$, proving that $\mathcal{U}$ is closed.
\end{proof}

\begin{lemma} \label{lem:minimizing}
Assume the boundary $\partial M$ is totally geodesic and $\Omega_x = 0$ for all $x \in \partial M$. Suppose the magnetic potential satisfies $\max_{x \in M} |\omega_x|_g < 1$ and every magnetic chord has Ma\~n\'e action $\pi$. Then every magnetic chord minimizes the free-time action functional.\label{lemma4.14}
\end{lemma}

\begin{proof}
Because $\max_{x \in M} |\omega_x|_g \le \kappa < 1$, the Lagrangian $L(x,v) + 1/2 = \frac{1}{2}|v|_g^2 - \omega_x(v) + \frac{1}{2}$ is positive on $TM$. Specifically, for any $v \in TM$, $L(x,v) + 1/2 \ge \frac{1}{2}(|v|_g^2 - 2\kappa|v|_g + 1) > 0$. This guarantees the energy level $E = 1/2$ is strictly above the Ma\~n\'e critical value $c$.

Let $x, y \in \partial M$ be distinct. To minimize the free-time action of curves from $x$ to $y$, we reparametrize curves to a fixed domain $[0,1]$ and consider the space of pairs $(T, \sigma) \in (0, \infty) \times H^1([0,1], M)$ such that $\sigma(0) = x$ and $\sigma(1) = y$. The free-time action functional on this manifold is given by
\begin{equation}
   {\mathcal A}(T, \sigma) = \int_0^1 \left( \frac{1}{2T}|\dot{\sigma}|_g^2 - \omega_{\sigma}(\dot{\sigma}) + \frac{T}{2} \right) ds.
\end{equation}
Because $E = 1/2 > c$, the free-time action functional is bounded below. Moreover, as proved in \cite{CIPP}, this inequality also ensures that for any minimizing sequence, the times $T$ are bounded away from zero and infinity. With $T$ constrained to a compact interval in $(0, \infty)$, the sequence of paths has uniformly bounded $H^1$-norms. Because $M$ is compact, the direct method of the calculus of variations guarantees the existence of a global minimum $(T, \sigma)$ yielding an absolutely continuous minimizing curve $\gamma: [0, T] \to M$.

Let $\gamma_0$ be a magnetic chord connecting $x\in \partial M$ to $y\in \partial M$ through the interior. By hypothesis, the Ma\~n\'e action is $A(\gamma_0) = \pi$, and since $\gamma_0$ is unit speed then its free-time action equals its Ma\~n\'e action so ${\mathcal A}(\gamma_0) = \pi$. If $\gamma$ is a candidate minimizer of the free-time action then ${\mathcal A}(\gamma)\le {\mathcal A}(\gamma_0)=\pi$.

There are two possibilities for the minimizer $\gamma$:
\begin{enumerate}
    \item $\gamma$ remains entirely within the boundary $\partial M$. Because $\omega$ vanishes on tangent vectors to the boundary, the Ma\~n\'e action of $\gamma$ reduces to the length of $\gamma$. By the rigidity of the scattering relation, the endpoints must be antipodal, $y = -x$, and $\partial M$ is a standard round hemisphere boundary, so the distance between antipodal points is $\pi$. Thus, $\pi\le A(\gamma)\equiv {\mathcal A}(\gamma) $.

    \item $\gamma$ enters the interior of $M$. Let $(t_1, t_2) \subseteq [0, T]$ be a maximal open interval where $\gamma$ lies strictly in the interior $M^\circ$. Because $\gamma$ is an unconstrained minimizer on this interval, it is a smooth solution to the Euler-Lagrange equations for the free-time action. Thus, the restriction $\gamma|_{(t_1, t_2)}$ is a unit-speed magnetic geodesic connecting boundary points; \emph{i.e.}, it is a magnetic chord. By hypothesis, the Ma\~n\'e action of this segment alone is $A\left (\gamma|_{(t_1, t_2)}\right ) =\pi$ and since this segment is a unit speed magnetic geodesic its free-time action is therefore ${\mathcal A}\left (\gamma|_{(t_1, t_2)}\right )=\pi$. Because the Lagrangian $L+1/2$ is strictly positive everywhere on $TM$, any additional portion of $\gamma$ outside the interval $(t_1, t_2)$ would add positive action. However, since the total action satisfies ${\mathcal A}(\gamma) \le \pi$, the remaining segments must have zero action, and therefore zero transit time. Thus, $(t_1, t_2) = (0, T)$, meaning $\gamma$ consists exactly of a single interior chord with ${\mathcal A}(\gamma) = \pi$.
\end{enumerate}

Hence the minimum of ${\mathcal A}$ over all curves joining boundary points is $\pi$, and every magnetic chord achieves exactly ${\mathcal A} = \pi$, so every magnetic chord is a global minimizer.
\end{proof}

\subsection{Proof of the main results}\label{section4.6}

Some remarks on terminology are in order. It is useful below to refer to solutions of equation \eqref{eq4.1} as ``$(g,\omega)$ magnetic geodesics'' and to write the Ma\~n\'e action as $A_{\omega}$. If a curve $\gamma$ in $M$ lifts to a curve in the tangent bundle with endpoints $(x,v)$ and $(y,w)$ we will simply say that ``$\gamma$ joins $(x,v)$ to $(y,w)$''.

\begin{proof}[Proof of Theorem \ref{thm:MagBangert}] By Proposition \ref{proposition4.12} we have $\mathcal{S}_{g,\Omega} = \mathcal{S}_{\rm can}$, and since the standard hemisphere is a reversible scattering system, it follows that our magnetic system is also \emph{boundary reversible}: $\mathcal{S}(-\mathcal{S}(x,v)) = (x,-v)$.

Let $(y, w) = \mathcal{S}(x,v)$. Let $\gamma$ be the $(g,\omega)$ magnetic geodesic from $(x,v)$ to $(y,w)$. By hypothesis, $A_\omega(\gamma) = \pi$.
Let $\sigma$ be the $(g,\omega)$ magnetic geodesic starting from $(y, -w)$. By boundary reversibility, $\sigma$ must exit exactly at $(x, -v)$. By the hypothesis on the Ma\~n\'e action, $A_\omega(\sigma) = \pi$. Let $\overline{\sigma}$ be the time reversal of $\sigma$. The curve $\overline{\sigma}$ travels from $x$ to $y$ and is a magnetic geodesic for the pair $(g, -\omega)$. Evaluated under the
$(g, \omega)$ system, its Ma\~n\'e action is $A_\omega(\overline{\sigma}) = T_\sigma + \int_\sigma \omega$.

Because $\gamma$ is a global minimizer of the action from $x$ to $y$ (Lemma \ref{lem:minimizing}), evaluating the reversed curve $\overline{\sigma}$ under the forward system yields the inequality
\begin{equation*}
    A_\omega(\gamma) \le A_\omega(\overline{\sigma}) \implies \pi \le T_\sigma +\int_\sigma \omega
\end{equation*}
Because $A_\omega(\sigma) = T_\sigma - \int_\sigma \omega = \pi$, we substitute $T_\sigma = \pi +\int_\sigma \omega$ to obtain
\begin{equation*}
    \pi \le \pi +2\int_\sigma \omega \implies \int_\sigma \omega \geq 0.
\end{equation*}
By symmetry, because $\sigma$ minimizes the action from $y$ to $x$, evaluating the time-reversal of $\gamma$ yields $\int_\gamma \omega \geq 0$.
We integrate this non-negative continuous function over $\partial_+ SM$ with respect to the standard measure $d\mu$. By Santal\'o's formula (which holds for non-trapping magnetic flows):
\begin{equation*}
    \int_{\partial_+ SM} \left( \int_{\gamma_{x,v}} \omega \right) d\mu = \int_{SM} \omega \, d\mu_{L} = 0
\end{equation*}
because $\omega$ is an odd function on $SM$. Therefore, $\int_\gamma \omega \equiv 0$ identically.

This implies $A_\omega(\overline{\sigma}) = T_\sigma = \pi$. Because $\overline{\sigma}$ achieves the minimum action $\pi$ required to cross from $x$ to $y$, it is also a global minimizer for the forward system. Consequently, it must satisfy the Euler-Lagrange equations for the forward system, meaning it is governed by the ODE for magnetic geodesics for $(g, \omega)$.
However, as a time-reversal of $\sigma$, the curve $\overline{\sigma}$ satisfies the ODE for the magnetic geodesics for the reversed system $(g, -\omega)$.

Equating the accelerations of both systems along the same curve yields $Y(v) = -Y(v)$ and this gives $Y(v) \equiv 0$. Since every $(x,v)\in SM$ is hit by some magnetic chord, the magnetic field $\Omega \equiv 0$ everywhere on $M$.
\end{proof}

\begin{proof}[Proof of Remark \ref{remark4.2}]
Let $(x,v) \in \partial_+ SM$. We have proved in Proposition \ref{prop:rs} that the magnetic geodesic $\gamma_{x,v}$ exits the manifold exactly at the antipodal point $-x$. For a unit-speed magnetic geodesic, the Ma\~n\'e action is:
\begin{equation*}
    A(\gamma_{x,v}) = \int_{\gamma_{x,v}} \left( |\dot{\gamma}|_g - \omega(\dot{\gamma}) \right) dt = L_g(\gamma_{x,v}) - \int_{\gamma_{x,v}} \omega=\pi,
\end{equation*}
which yields $L_g(\gamma_{x,v}) = \pi + \int_{\gamma_{x,v}} \omega$. By hypothesis the Riemannian length of any curve connecting $x$ to $-x$ is bounded below by $\pi$, so $L_g(\gamma_{x,v}) \ge \pi$. Thus:
\begin{equation} \label{eq:action_ineq}
    \pi + \int_{\gamma_{x,v}} \omega \ge \pi \implies \int_{\gamma_{x,v}} \omega \ge 0
\end{equation}

We now integrate the 1-form $\omega$ over the entire unit tangent bundle $SM$ with respect to the Liouville measure $d\mu_L$. By assumption, the magnetic flow is non-trapping and thus we can apply Santal\'o's formula to obtain:
\begin{equation}
  0=  \int_{SM} \omega \, d\mu_L = \int_{\partial_+ SM} \left( \int_{\gamma_{x,v}} \omega \right) d\mu.\label{eq:w=0}
\end{equation}
Since equation \eqref{eq:action_ineq} establishes that the inner line integral $\int_{\gamma_{x,v}} \omega \ge 0$ everywhere on $\partial_+ SM$, \eqref{eq:w=0} forces
\begin{equation*}
    \int_{\gamma_{x,v}} \omega \equiv 0 \quad \text{for all } (x,v) \in \partial_+ SM
\end{equation*}
This implies that for every such chord, the action achieves the Riemannian lower bound: $A(\gamma_{x,v}) = L_g(\gamma_{x,v}) = \pi$. Because $d_g(x,-x) = \pi$, these magnetic chords are also length-minimizing Riemannian geodesics. Consequently, they satisfy both the magnetic and the Riemannian geodesic equations, which forces $\Omega$ to vanish.
\end{proof}

\section{Asymptotically flat spacetimes} \label{section5}
\setcounter{equation}{0}

\noindent Consider $(n+1)$-dimensional Minkowski spacetime $({\mathbb M}, g)$, $n\ge 2$. It is well known that there is an embedding $\phi$ of $(M,g)$ into the \emph{Einstein static universe} ${\mathbb R}\times {\mathbb S}^n$ with metric ${\bar g}=-dt^2+g({\mathbb S}^n,{\rm can})$ such that $\phi^* {\bar g}$ is conformally isometric to $g$. The closure of the image of $({\mathbb M}, g)$, together with the induced metric ${\bar g}$, yields a conformal compactification of $({\mathbb M}, g)$. As explained in standard texts such as \cite{Wald}, the boundary of $M$ in ${\bar M}$ consists of null hypersurfaces denoted by ${\mathscr{I}}^{\pm}$ (called future and past null infinity) as well as three points $i^{\pm}$ and $i^0$, the last of which denotes \emph{spatial infinity} and serves to one-point-compactify every Cauchy surface in ${\mathbb M}$, each of which then becomes an $n$-sphere.

Future and past null infinity are ruled by null geodesics of ${\bar g}$. These are called the \emph{generators} of ${\mathscr{I}}^{\pm}$. They have a peculiar property. Consider some point $p\in {\mathscr{I}}^-$ and consider the congruence of all null geodesics which depart from $p$. One such geodesic is a generator of ${\mathscr{I}}^-$, call it $\psi^-$, which reaches $i^0$ and can be extended by concatenation with a null generator $\psi^+$ of ${\mathscr{I}}^+$ passing through $q$. All other geodesics are contained entirely in $M$ except for their common past endpoint $p\in {\mathscr{I}}^-$ and their common future endpoint $q\in{\mathscr{I}}^+$. As previously, we call $q$ the \emph{spacetime antipode} to $p$, though now it is not defined solely (or at all) by null geodesics that remain on ${\mathscr{I}}$. The concatenation of $\psi^-$ and $\psi^+$, which we call $\psi$ and to which we give a ${\bar g}$-affine parametrization, is a smooth null geodesic in the Einstein static universe but lies entirely on the conformal infinity of the conformally embedded Minkowski spacetime.

No point $q'$ of $\psi^+$ lying to the causal past of $q$ (as defined by the Einstein static universe metric) can be reached from $p$ by any null geodesic, or any causal curve, except $\psi$. That is, no causal curve in $M$, or containing at least one point of $M$, extends to join $p$ to any such $q'$. Remarkably, every point along every other null generator of ${\mathscr{I}}^+$ can be reached from $p$ by causal, indeed timelike, curves that lie entirely in $M$ except for their endpoints \cite{PSW, Cameron}. In terms of the chronological future of $p$, we write that $I^+(p,{\bar M})$ contains all of ${\mathscr{I}}^+$ except for the segment of $\psi^+$ that lies to the causal past of $q$. In fact, $q$ is conjugate to $p$ along any null geodesic that begins at $p$, making $(p,q)$ a pair of \emph{conjugate points at infinity}.

\begin{lemma}[Minkowski spacetime antipodes]\label{lemma5.1}
Let ${\bar {\mathbb M}}$ denote the closure of the image in the Einstein static universe of Minkowski spacetime ${\mathbb M}$, with the usual conformally isometric embedding. For each $p\in{\mathscr{I}}^-$, there is a spacetime antipode $q\in{\mathscr{I}}^+$ and a null geodesic generator $\psi^+$ of ${\mathscr{I}}^+$ through $q$ such that
\begin{itemize}
\item [(i)] if $q'\in \psi^+\cap J^-(q,{\bar {\mathbb M}})$ and $q'\neq q$ then no causal curve in $M$ extends from $p$ to $q'$, and
\item [(ii)] if $q'$ lies in the complement in ${\mathscr{I}}^+$ of $\psi^+\cap J^-(q,{\bar {\mathbb M}})$ there is always a timelike curve extending from $p$ to $q'$.
\end{itemize}
Moreover,
\begin{itemize}
\item [(iii)] every null geodesic from $p$ ends at $q$.
\end{itemize}
\end{lemma}

We omit the proof, which is an exercise.
Now we ask the following.

\begin{question}\label{question5.2}
For $(M,g)$ an asymptotically flat, globally hyperbolic spacetime, say that for every $p\in{\mathscr{I}}^-$, every null geodesic from $p$ ends at $q$ and no timelike curve from $p$ ends at $q$. Is the spacetime conformally isometric to Minkowski spacetime?
\end{question}

\begin{remark}\label{remark5.3} \
\begin{enumerate}
\item An asymptotically flat spacetime in which every point of ${\mathscr{I}}^+$ is in the chronological future of each point of ${\mathscr{I}}^-$ is said to obey the \emph{null boundary Penrose property} \cite{CD}. Minkowski spacetime does not obey this property.
\item There are globally hyperbolic spacetimes not isometric to Minkowski spacetime which do not obey the null boundary Penrose property. Globally hyperbolic 4-dimensional spacetimes that approach a negative mass Schwarzschild spacetime outside of a timelike tube and are smooth everywhere provide examples. Such spacetimes clearly exist but will not be vacuum and will violate energy conditions \cite{PSW}. They obey parts (i) and (ii) of Lemma \ref{lemma5.1} but not part (iii).
\end{enumerate}
\end{remark}

When we assume that the spacetime metric admits an asymptotically flat metric in its conformal class that obeys the null energy condition, the proof of Proposition \ref{proposition3.1} can be adapted to the asymptotically flat setting as follows. As in subsection \ref{subsection3.1}, we take $g$ to obey the null energy condition, and show that it is Minkowski. Since we have assumed that the spacetime metric is conformally isometric to $g$, the spacetime will be conformally isometric to Minkowski spacetime.

\begin{proposition}[Null energy rigidity for asymptotic flatness]\label{proposition5.4}
Let $(M,g)$ be an asymptotically simple, four-dimensional asymptotically flat spacetime. Assume that the null energy condition \eqref{eq3.1} holds for all null vectors $\ell\in T_rM$ and all $r\in M$. Assume also that for each $p\in {\mathscr{I}}^-$, every null geodesic in spacetime with idealized past endpoint $p$ has idealized future endpoint $q$ and no causal curve in spacetime meets ${\mathscr{I}}^+$ to the past of $q$.\footnote
{The phrase ``and no causal curve in spacetime meets ${\mathscr{I}}^+$ to the past of $q$'' seems necessary to preclude the possibility that there is a sequence of timelike curves in the past of $q$ which converges to the generator of ${\mathscr{I}}$ from $p$ to $q$.}
Then $(M,g)$ is conformal to Minkowski spacetime.
\end{proposition}

\begin{proof}
We need only minor modifications to the proof of Proposition \ref{proposition3.1}. We now first pick an arbitrary point $r\in M$ and an arbitrary future-null direction there. This defines a null geodesic $\gamma$ which extends to some $p\in {\mathscr{I}}^-$ in the past and to some $q\in{\mathscr{I}}^+$ in the future. But then $p$ and $q$ are spacetime antipodes, and as before no null geodesic from $p$ to $q$ has conjugate points. Then, as before, we conclude that $g$ is Einstein and locally conformally flat. By asymptotic flatness, the Einstein constant must be zero in this case.
\end{proof}

However, if we replace the assumption of the null energy condition with the assumption that $g$ is static or, more generally, stationary, the optical metric in the asymptotically flat case is a complete metric on a noncompact manifold. We do not have a manifold-with-boundary, as we had in the asymptotically AdS case, and minimizing geodesics of the optical metric have infinite lengths. The theorem of \cite{Bangert} and the arguments of our Section \ref{section4} do not apply.

\appendix

\section{Null geodesics in stationary and static spacetimes}\label{appendix}
\setcounter{equation}{0}

\noindent We work in local coordinates. Then from Definition \ref{definition2.3}, an arbitrary stationary metric on a globally hyperbolic spacetime $(M,g)$ can be written as
\begin{equation}
\label{eqA.1}
\begin{split}
ds^2 =&\, -V \left ( dt+\omega_idy^i\right )^2 +h_{ij}dy^idy^j\\
=&\,-Vdt^2-2V\omega_idy^idt+\left ( h_{ij}-V\omega_i\omega_j\right ) dy^idy^j
\end{split}
\end{equation}
where $\frac{\partial}{\partial t}$ is a Killing vector field \cite[Equation 4.3.3, p 149]{Chrusciel}. The component functions are all independent of $t$. If one takes the quotient by identifying points along integral curves of this Killing field, the quotient manifold is the Riemannian manifold $(\Sigma,h)$.

Let $v$ denote the tangent vector field to a future-null geodesic in spacetime. Writing the geodesic equation in local coordinates, we have
\begin{equation}
\label{eqA.2}
g_{ab}\frac{dv^b}{ds}+[bc,a]v^bv^c=0,
\end{equation}
where $[bc,a]:=g_{ad}\Gamma^d_{bc}$ denotes the Christoffel symbols of the first kind. Those symbols are
\begin{equation}
\label{eqA.3}
\begin{split}
[00,0]=&\, 0,\\
[0i,0]=&\, [i0,0]=-\frac12 \partial_i V,\\
[ij,0]=&\, -\frac12\left [ \partial_i \left ( V\omega_j\right ) +\partial_j \left (V\omega_i\right ) \right ]\\
[00,i]=& \frac12 \partial_i V,\\
[0j,i]=&\, [j0,i]=\frac12\left [ \partial_i \left ( V\omega_j\right ) -\partial_j \left ( V\omega_i \right ) \right ],\\
[jk,i]=&\, \frac12 \left [ \partial_j \left ( h_{ki} -V\omega_k\omega_i\right ) +\partial_k \left ( h_{ij} -V\omega_i\omega_j \right )\right . \\
&\, \left . -\partial_i \left ( h_{jk} -V\omega_j\omega_k\right ) \right ],
\end{split}
\end{equation}
where $0$ denotes the subspace spanned by $\frac{\partial}{\partial t}$ and $i$, $j$, and $k$ run over the orthogonal complement (so $i\neq 0$, $j\neq 0$, and $k\neq 0$).

The component of the geodesic equation along $\frac{\partial}{\partial t}$ gives
\begin{equation}
\label{eqA.4}
\begin{split}
0=&\, -V\frac{dv^0}{ds}-V\omega_i\frac{dv^i}{ds}-Vv^iv^j\partial_i\omega_j-v^iv^0\partial_iV -v^i\omega_iv^j\partial_jV,\\
=&\, -\frac{d}{ds}\left ( V \left ( v^j\omega_j +v^0\right ) \right ],
\end{split}
\end{equation}
where we've used the chain rule $v^i \partial_i=\frac{d}{ds}$ when differentiating ($t$-independent) metric component functions along integral curves. Hence
\begin{equation}
\label{eqA.5}
V \left ( v^j\omega_j +v^0\right )=E,
\end{equation}
where $E$ is constant along each integral curve. Moreover, the integral curves are null, so they obey
\begin{equation}
\label{eqA.6}
0=V\left ( v^i\omega_i+v^0\right )^2-h(v,v).
\end{equation}
Hence $h(v,v):=h_{ij}v^iv^j=\varepsilon^2/V$, and $\varepsilon =0$ if and only if the curves are trivial.

The components of the geodesic equation in directions orthogonal to $\frac{\partial}{\partial t}$ yield
\begin{equation}
\label{eqA.7}
\begin{split}
0=&\, -V\omega_i \frac{dv^0}{ds} +\frac12 v^0v^0\partial_i V - \left [ \omega_i\partial_j V -\omega_j \partial_i V\right ]v^0v^j +\left ( h_{ij} -V\omega_i\omega_j \right ) \frac{dv^j}{ds} \\
&\, +V\left ( \partial_i \omega_j -\partial_j \omega_i\right ) v^0v^j +\frac12 \left [ \partial_j \left ( h_{ki} \right ) +\partial_k \left ( h_{ij} \right ) -\partial_i \left ( h_{jk} \right )\right ]v^jv^k\\
&\, -V\omega_iv^jv^k\partial_j\omega_k - \omega_iv^j\omega_j v^k\partial_k V +\frac12 v^j\omega_jv^k\omega_k \partial_i V +Vv^k\omega_k v^j \left ( \partial_i \omega_j - \partial_j\omega_i\right )\\
=&\, -V\omega_i \left [ \frac{dv^0}{ds} +\omega_j \frac{dv^j}{ds} +v^jv^k\partial_j\omega_k \right ] -\omega_i\partial_j V \left [ v^0v^j +v^k\omega_k v^j\right ]\\
&\, +\partial_i V \left [ \frac12 v^0v^0 +v^j\omega_j v^0 +\frac12 \left (v^j\omega_j\right )^2 \right ] \\
&\, + h_{ij}  \frac{dv^j}{ds} +\frac12 \left ( \partial_j h_{ki} +\partial_k h_{ij} -\partial_i h_{jk} \right )v^jv^k \\
&\, +V\left [ v^k\omega_k+v^0\right ] v^j \left ( \partial_i \omega_j- \partial_j\omega_i\right ) .
\end{split}
\end{equation}
The first line in the last expression vanishes by using the top line of \eqref{eqA.4}. Using \eqref{eqA.6}, we see that the terms in square brackets in the second line become simply $\frac{1}{2V}h(v,v)$. Furthermore, if $D$ is the Levi-Civita connection compatible with $h$, the terms in the third line reduce to $h_{ij} \left ( D_v v\right )^j$. Finally, the fourth line simplifies using \eqref{eqA.6}, so we obtain
\begin{equation}
\label{eqA.8}
0 = h_{ij}(D_v v)^j + \frac{1}{2V}h(v,v)\,\partial_i V
    + \sqrt{V h(v,v)}\, v^j\left(\partial_i\omega_j-\partial_j\omega_i\right).
\end{equation}
Contracting against $v^i$, we can check that $|v|=const$.

Finally, define the \emph{optical metric} ${\tilde h}_{ij}:=\frac{1}{V} h_{ij}$, and let ${\tilde D}$ denote its Levi-Civita connection. Furthermore, let $x^i(t)$ denote the components of the projection into the optical manifold of the null geodesic, so that $v^i=\frac{dx^i}{dt}$. Reparametrizing this curve by $t\mapsto{\tilde t}=\int\frac{dt}{V\circ x(t)}$, we obtain $v^i=\frac{1}{V}{\tilde v}^i$ where ${\tilde v}^i=\frac{dx^i}{d{\tilde t}}$. Then it is easy to see that \eqref{eqA.8} reduces to
\begin{equation}
\label{eqA.9}
\begin{split}
0= &\, {\tilde h}_{ij}\left ( {\tilde D}_{\tilde v} {\tilde v}\right )^j +\sqrt{{\tilde h}({\tilde v},{\tilde v)}} {\tilde v}^j \left ( \partial_i \omega_j- \partial_j\omega_i\right )\\
= &\, {\tilde h}_{ij}\left ( {\tilde D}_{\tilde v} {\tilde v}\right )^j +\varepsilon{\tilde v}^j \left ( \partial_i \omega_j- \partial_j\omega_i\right ),
\end{split}
\end{equation}
where we use that ${\tilde h}({\tilde v},{\tilde v})=Vh(v,v)=\varepsilon^2$. We can write this result as
\begin{equation}
\label{eqA.10}
{\tilde h}_{ij}\left ( {\tilde D}_{\tilde v} {\tilde v}\right )^j = \varepsilon \Omega_{ji}{\tilde v}^j = \varepsilon \left ( \partial_j\omega_i-\partial_i\omega_j\right ){\tilde v}^j,
\end{equation}
where we introduce the quantity $\Omega=d\omega$.

Integral curves of solutions ${\tilde v}$ of this equation in Riemannian manifolds are called \emph{magnetic geodesics} of $(\Sigma,{\tilde h}, \omega)$. Equation \ref{eqA.10} is the Euler-Lagrange equation for the free-time action defined by $(\Sigma,{\tilde h}, \omega)$. Furthermore, all solutions of the Euler-Lagrange equation for the Ma\~n\'e action defined by $(\Sigma,{\tilde h}, \omega)$ are magnetic geodesics up to parametrization; i.e., they solve \eqref{eqA.10} after a reparametrization if necessary.

\section{Two lemmata used in Section \ref{section3}}\label{appendixB}

\begin{proof}[Proof of Lemma \ref{lemma3.2}]
Fix an orthonormal basis $\left \{ e_0,e_1,\dots,e_{n-1}\right \}$ such that $e_0$ is timelike. Then $e_0\pm e_1$ is null, so
\begin{equation}
\label{eqB.1}
\begin{split}
0=&\, T(e_0+e_1,e_0+e_1)= T(e_0,e_0)+T(e_1,e_1)+2T(e_0,e_1),\\
0=&\, T(e_0-e_1,e_0-e_1)= T(e_0,e_0)+T(e_1,e_1)-2T(e_0,e_1).
\end{split}
\end{equation}
Subtracting one of these from the other, we get
\begin{equation}
\label{eqB.2}
T(e_0,e_1)=0,
\end{equation}
and thus
\begin{equation}
\label{eqB.3}
T(e_0,e_0)=-T(e_1,e_1)
\end{equation}
Replacing $e_1$ by $e_2$ and so forth up to $e_{n-1}$, we see that
\begin{equation}
\label{eqB.4}
\begin{split}
T(e_0,e_i)=&\, 0\text{ for }i=1,2,\dots,n-1,\text{ and }\\
T(e_0,e_0)=&\, -T(e_1,e_1)=-T(e_2,e_2)=\dots=-T(e_{n-1},e_{n-1}).
\end{split}
\end{equation}

Next, $e_0+\frac{1}{\sqrt{2}}\left ( e_1+e_2\right )$ is also null, so (using that $T(e_0,e_1)=0$ and $T(e_0,e_2)=0$) we have that
\begin{equation}
\label{eqB.5}
\begin{split}
0=&\, T(e_0,e_0)+\frac12 T(e_1+e_2,e_1+e_2)\\
=&\, T(e_0,e_0)+\frac12 \left [T(e_1,e_1)+T(e_2,e_2) +2T(e_1,e_2) \right ].
\end{split}
\end{equation}
But $T(e_0,e_0)=-T(e_1,e_1)=-T(e_2,e_2)=-\frac12 \left [ T(e_1,e_1)+T(e_2,e_2)\right ]$, so \eqref{eqB.4} yields $T(e_1,e_2)=0$. Cycling through the basis vectors as before, we have
\begin{equation}
\label{eqB.6}
T(e_i,e_j)=0\quad \forall i,j\in \{ 1,\dots,n-1\}\text{ such that }i\neq j.
\end{equation}
Then from \eqref{eqB.4} and \eqref{eqB.6} we have that $T$ is proportional to $g$.
\end{proof}

\begin{proof}[Proof of Lemma \ref{lemma3.3}]
Define the complex null basis $\left \{ \ell, k, m, {\bar m}\right \}$ such that each basis vector is null, $\ell ={\bar \ell}$ and $k={\bar k}$ are real, $g(\ell,k)=-1$, and $g(m,{\bar m})=1$, and products of the other basis vectors vanish. Define the Weyl scalars
\begin{equation}
\label{eqB.7}
\begin{split}
\Psi_0=&\, C(\ell,m,\ell,m),\\
\Psi_1=&\, C(\ell,k,\ell,m),\\
\Psi_2=&\, C(\ell,m,{\bar m},k),\\
\Psi_3=&\, C(\ell,k,{\bar m},k),\\
\Psi_4=&\, C(k,{\bar m},k,{\bar m}).
\end{split}
\end{equation}
Then $C$ can be expressed as a linear combination of products of these basis vectors, with coefficients given by the scalars $\Psi_i$, $i=0,\dots,4$, and their complex conjugates (see for example \cite[p 44, Equation (298)]{Chandra}), so $C$ vanishes when all of the Weyl scalars vanish.

Let $a\in {\mathbb C}$. We can define a family of complex bases obtained from the one above by null rotation:
\begin{equation}
\label{eqB.8}
\begin{split}
\ell(a):=&\, \ell +{\bar a} m+a{\bar m} +|a|^2 k,\\
k(a):=&\, k,\\
m(a):=&\, m+ak,\\
{\bar m}(a):=&\, {\bar m}+{\bar a}k.
\end{split}
\end{equation}
Then $\left \{ \ell(a), k(a), m(a),{\bar m}(a)\right \}$ is still a complex null basis for each choice of $a\in {\mathbb C}$, with $g(\ell(a),k(a))=-1$, $g(m(a),{\bar m}(a))=1$, etc. By hypothesis,
\begin{equation}
\label{eqB.9}
C(\ell(a),m(a),\ell(a),m(a))=0
\end{equation}
for each $a\in {\mathbb C}$.
Expand \eqref{eqB.9} in powers of $a$ using \eqref{eqB.8} and simplify the result using the conditions given in the definition of an algebraic Weyl tensor. This yields
\begin{equation}
\label{eqB.10}
\begin{split}
0=&\, C(\ell(a),m(a),\ell(a),m(a))\\
=&\, C(\ell,m,\ell,m) +4aC(\ell,k,\ell,m)+6a^2C(\ell,m,{\bar m},k)\\
&\, + 4a^3C(\ell,k,{\bar m},k) +a^4C(k,{\bar m},k,{\bar m})\\
=&\, \Psi_0 +4a\Psi_1+6a^2\Psi_2+4a^3\Psi_3+a^4\Psi_4.
\end{split}
\end{equation}
Since this must hold for all $a\in{\mathbb C}$, we obtain that $\Psi_i=0$, for $i=0,\dots,4$, and so $C$ vanishes identically.
\end{proof}

\subsection*{Data availability statement} This manuscript has no associated data.
\subsection*{Conflict of interest statement} The authors have no conflicts of interest.

\end{document}